\documentclass[11pt]{amsart}%
\usepackage{xypic}\xyoption{all}
\usepackage{amsthm,euscript}
\usepackage{paralist} \usepackage{amsopn}
\usepackage{amsfonts}
\usepackage{amsmath}
\usepackage{latexsym}
\usepackage{amscd}
\usepackage{amssymb}
\usepackage{amsmath}
\usepackage{fancyhdr}
\usepackage{setspace}
\usepackage{bm}
\usepackage[colorlinks=true, pdfstartview=FitV, linkcolor=blue,
citecolor=blue,
urlcolor=blue,breaklinks=true]{hyperref}
\usepackage{fancyhdr}

\newcommand{\sm}{\smallskip}
\newcommand{\ms}{\medskip}

\long\def\ignore#1\endignore{\relax}

\swapnumbers

\newtheorem{theorem}[subsection]{Theorem}

\newtheorem{proposition}[subsection]{Proposition}

\newtheorem{lemma}[subsection]{Lemma}
\newtheorem{Lemma}[subsection]{Lemma}
\newtheorem{corollary}[subsection]{Corollary}

\theoremstyle{definition}

\newtheorem{remark}[subsection]{Remark}

\numberwithin{equation}{section} \allowdisplaybreaks

\numberwithin{equation}{subsection}

\newcommand\inpr{(\cdot|\cdot)} 
\newcommand\lan{\langle} \newcommand\ran{\rangle}

\newcommand\ind{_{\rm ind}}

\newcommand\ts{\textstyle}

\newcommand\gr{{\rm gr}}
\newcommand\ch{\sp{\scriptscriptstyle\vee}}

\newcommand\co{\colon} 

\newcommand\ot{\otimes}
\newcommand\pa{\partial}
\newcommand\an{^{\rm an}}

\newcommand{\mn}{^{-1}}

\newcommand{\dbl}{[\!\![}
\newcommand{\dbr}{]\!\!]}

\newcommand{\Aut}{\operatorname{Aut}}
\DeclareMathOperator{\ad}{ad}

\newcommand{\Hom}{\operatorname{Hom}}
\newcommand{\pr}{\operatorname{pr}}

\newcommand{\End}{\operatorname{End}}

\newcommand{\Ctd}{\operatorname{Ctd}}
\newcommand{\SCDer}{\operatorname{SCDer}}
\newcommand{\Lie}{\operatorname{Lie}}
\newcommand{\EL}{{\operatorname{EL}}}
\newcommand{\GL}{{\operatorname{GL}}}

\newcommand{\Int}{{\operatorname{Int}}}

\newcommand{\Mat}{{\operatorname{M}}}

\newcommand{\Der}{\operatorname{Der}}

\newcommand{\Tr}{\operatorname{Tr}}

\DeclareMathOperator{\EAut}{EAut}
\DeclareMathOperator{\CDer}{CDer}
\DeclareMathOperator{\rmE}{E} 
\DeclareMathOperator{\ev}{ev}
\DeclareMathOperator{\EA}{EA}
\DeclareMathOperator{\IDer}{IDer}
\DeclareMathOperator{\IE}{IE}

\DeclareMathOperator{\Span}{span}
\DeclareMathOperator{\SDer}{SDer}

\DeclareMathOperator{\supp}{supp}

\def\gg{\mathfrak g}

\newcommand{\g}{\mathfrak g}

\newcommand\frh{\ensuremath{\mathfrak{h}}} \newcommand\h{\frh}
\newcommand\lsl{\ensuremath{\mathfrak{sl}}}
\newcommand\gl{\ensuremath{\mathfrak{gl}}}

\newcommand{\bbC}{{\mathbb C}}

\newcommand\NN{\mathbb{N}}
\newcommand\ZZ{\mathbb{Z}}

\def\al{\alpha}\def\alp{\al}
\def\be{\beta}
\def\ga{\gamma}
\def\lm{\lambda}
\def\la{\lm}

\def\si{\sigma}
\newcommand\De{\Delta}
\newcommand\eps{\varepsilon}

\newcommand\La{\Lambda} 

\newcommand\ta{\tau}
\newcommand\vphi{\varphi}

\newcommand\boldeps{{\boldsymbol{\varepsilon}}}
\newcommand\beps{\boldeps}

\newcommand{\euD}{\EuScript{D}}

\newcommand\euQ{\EuScript{Q}}

\newcommand{\euZ}{\EuScript{Z}}

\newcommand\scA{\mathcal{A}}
\newcommand\scB{\mathcal{B}}
\newcommand\scC{\mathcal{C}}
\newcommand\scD{\mathcal{D}}

\newcommand\scQ{\mathcal{Q}}

\newcommand\scZ{\euZ}

\def\bG{\text{\rm \bf G}}

\begin{document}

\title[]{Conjugacy of Cartan subalgebras in EALAs with a non-fgc centreless
core}

\author{V. Chernousov}
\address{Department of Mathematics, University of Alberta,
    Edmonton, Alberta T6G 2G1, Canada}
\thanks{ V. Chernousov was partially supported by the Canada Research
Chairs Program and an NSERC research grant} \email{chernous@math.ualberta.ca}

\author{E. Neher}
\address{Department of Mathematics and Statistics, University of Ottawa,
    Ottawa, Ontario K1N 6N5, Canada}
\thanks{E.~Neher was partially supported by a Discovery grant from NSERC}
\email{neher@uottawa.ca}

\author{A. Pianzola}
\address{Department of Mathematics, University of Alberta,
    Edmonton, Alberta T6G 2G1, Canada.
    \newline
 \indent Centro de Altos Estudios en Ciencia Exactas, Avenida de Mayo 866, (1084) Buenos Aires, Argentina.}



\thanks{A. Pianzola wishes to thank NSERC and CONICET for their
continuous support}\email{a.pianzola@gmail.com}

\thanks{The second author wishes to thank the Department of Mathematical Sciences
at the University of Alberta for hospitality during part of the work on this
paper.}

\begin{abstract}  We establish the conjugacy of Cartan subalgebras for
extended affine Lie algebras whose centreless core is ``of type A'', i.e.,
$\ell \times \ell$  matrices over a quantum torus $\scQ$ whose trace
lies in the commutator space of $\scQ$. This settles the last outstanding
part of the conjugacy problem for Extended Affine Lie Algebras that remained
open.
\end{abstract}

\maketitle

\noindent \hfill {\em  Dedicated to E. B. Vinberg on the occasion of his 80th
birthday } \medskip

\section*{Introduction}

This work is the last of a series of papers \cite{CGP, CNP, CNPY} devoted to
proving the Conjugacy Theorem for Extended Affine Lie Algebras: \sm

{\bf Conjugacy Theorem.} {\it Let $(E,H)$ and $(E,H')$ be two extended affine
Lie algebras, both defined on the same underlying Lie algebra $E$  over an
algebraically closed field of characteristic $0$. Then there exists an
automorphism $f$ of $E$ such that $f(H) = H'$.} \ms

Conjugacy has been established for all but one family of EALAs, and it is this
remaining case that our paper settles. Below we give a brief historical account
of the ``Conjugacy problem".

Let $\gg$ be a finite-dimensional split simple Lie algebra over a field $k$ of
characteristic $0$, and let $\bG$ be the simply connected Chevalley-Demazure
algebraic group associated to $\gg$. Chevalley's theorem (\cite[VIII, \S3.3,
Cor. de la Prop.~10]{Bbk}) asserts that all split Cartan subalgebras $\h$ of
$\gg$ are conjugate under the adjoint action of $\bG(k)$ on $\gg.$ This is one
of the central results of classical Lie theory. One of its immediate
consequences is that the corresponding root system is an invariant of the Lie
algebra (i.e., it does not depend on the choice of Cartan subalgebra).

We now look at the analogous question in the infinite dimensional set up as it
relates to extended affine Lie algebras (EALAs for short). Even if the field
$k$ is assumed to be algebraically closed, the reader should keep in mind that
our results are more akin to the setting of Chevalley's theorem for general $k$
than to conjugacy of Cartan subalgebras in finite-dimensional simple Lie
algebras over algebraically closed fields. The role of $(\gg, \h)$ is now
played by a pair $(E,H)$ consisting of a Lie algebra $E$ and a ``Cartan
subalgebra" $H$. There are other Cartan subalgebras in $E$, and the question is
whether they are conjugate and, if so, under the action of which group.

The first example is that of untwisted affine Kac-Moody Lie algebras. Let $R =
k[t^{\pm 1}]$. Then
\begin{equation}\label{KacMoody}
E = \gg \ot_k R  \oplus kc \oplus kd
\end{equation}
and \begin{equation}\label{cartan}
H = \mathfrak{h} \ot 1 \oplus kc \oplus kd.
\end{equation}

The relevant information is as follows. The $k$-Lie algebra $\gg \ot_k R \oplus kc$ is a central extension (in fact the universal central extension) of the $k$-Lie algebra $\gg \ot_k R$. The derivation $d$ of $\gg \ot_k R$ corresponds to the degree derivation $t d/dt$ acting on $R$. Finally $\mathfrak{h}$ is a fixed Cartan subalgebra of $\gg.$ The nature of $H$ is that it is abelian, it acts $k$-diagonalizably on $E$, and it is maximal with respect to these properties. Correspondingly, these subalgebras are called MADs (Maximal Abelian Diagonalizable) subalgebras. A celebrated theorem of Peterson and Kac \cite{PK} states that all MADs of $E$ are conjugate (under the action of a group that they construct which is the analogue of the simply connected group in the finite-dimensional case). Similar results hold for the twisted affine Lie algebras. These algebras are of the form
$$ E = L \oplus kc \oplus kd. $$
The Lie algebra $L$ is a loop algebra $L = L(\gg, \si)$  for some finite order
automorphism $\si$ of $\gg$ (see \cite{K} for details). If $\si$ is the
identity, we are in the untwisted case. The ring $R$ can be recovered as the
centroid of $L$.

Extended affine Lie algebras can be thought of as multi-variable
generalizations of finite-dimensional simple Lie algebras and affine Kac-Moody
algebras. For example, taking $R=k[t_1^{\pm 1}, \ldots, t_n^{\pm 1}]$ in
\eqref{KacMoody} and increasing $kc$ and $kd$ correspondingly in the obvious
way leads to toroidal algebras, an important class of examples of EALAs. But as
is already the case for affine Kac-Moody algebras, there are many interesting
examples of EALAs where $\g\ot_k R$ is replaced by a more general algebra, a
so-called Lie torus (see \ref{def:lietor}).

In the EALA set up, the Lie algebras $\gg$ as above are the case of nullity $n
= 0$, while the affine Lie algebras are the case of nullity $n = 1$. In higher
nullity $n$ we have $R = k[t_1^{\pm 1},\ldots,t_m^{\pm 1}]$ for some $m \leq
n,$ where again $R$ is the centroid of the centreless core $E_{cc}$ of the
given EALA. The theory of EALAs divides naturally into two cases:

(a) $m = n$.  In this case $E_{cc}$ is a module of finite type over the
centroid $R.$ It is refereed to as the ``fgc case" (short for finitely
generated over the centroid). If $k$ is algebraically closed,\footnote{See Remark \ref{fieldnature} below} the $R$-Lie algebra $E_{cc}$ is a multiloop algebra based on a (unique) $\gg$ as above. In particular $E_{cc}$ is twisted form of
$\gg \otimes_k R.$ This fact allows the
powerful methods of descent theory and reductive group schemes to be used.
Conjugacy at the level of $E_{cc}$ was established in \cite{CGP}. The lift of
this conjugacy theorem from $E_{cc}$ to $E$ is the main result of
\cite{CNPY}.

(b) $m < n.$  This is the so-called non-fgc case. Now $E_{cc}$ is not a module
of finite type over its centroid and $E_{cc}$ is not a twisted form of $\gg
\otimes_k R.$ The non-abelian Galois cohomology methods used in (a) are not
available. Fortunately, in the non-fgc case the nature of $E_{cc}$ is fully
understood. Indeed $E_{cc} = \lsl_\ell(\scQ)$ for some quantum torus  $\scQ$
and positive integer $\ell$ (see below for details). Conjugacy at the level of
$E_{cc}$ was established in \cite{CNP} by means of a ``specialization" trick of
its own interest. The main result of the present paper is the lift of conjugacy
for $E_{cc}$ to $E$ in the non-fgc case. This completes the proof that
``Conjugacy of Cartan subalgebras" holds for all EALAs.

The canonical procedure that associates to an EALA $E$ its core $E_c$ and
centreless core $E_{cc}=E_c / \scZ(E_c)$ can be reversed in the sense that one
can re-construct $E$ from its centreless core $E_{cc}$ by a special type of a
$2$-fold extension (in this paper we generalize this to so-called ``interlaced
extensions''). Moreover, going from $E$ to $E_{cc}$ is also a well-behaved
procedure at the level of the Cartan subalgebras: Consider $H_c = H \cap E_c$
and let $\pi \co E_c \to E_{cc}$ be the canonical map, then $H_{cc}= \pi(H_c)$
and the analogously defined $H'_{cc}$ are special types of MADs in $E_{cc}$.
Even more, every automorphism $f$ of $E$ leaves $E_c$ and hence also
$\scZ(E_c)$ invariant and so gives rise to an automorphism $f_{cc}$ of
$E_{cc}$. Thus, if our Main Theorem holds, then necessarily there exists some
automorphism $f_{cc} \in \Aut_k(E_{cc})$ such that $f_{cc} (H_{cc}) = H'_{cc}$.
From  this perspective, our approach of proving conjugacy ``upstairs'' on the
EALA level is the most natural one: we want to show that

\begin{enumerate}[(A)]
  \item there exists $f_{cc} \in \Aut_k(E_{cc})$ satisfying $f_{cc}
      (H_{cc}) = H'_{cc}$, and

 \item the automorphism $f_{cc}$ of (A) can be ``lifted'' to an
     automorphism $f$ of $E$ such that $f(H) = H'$.

\end{enumerate} \sm

Problem (A) has been solved in the two papers \cite{CGP} (the fgc case) and
\cite{CNP} (the non-fgc case).\sm

This leaves us with problem (B). Its difficulty lies in the fact
that a lift $f\in \Aut(E)$ of $f_{cc}$ (if it exists at all) will not
necessarily satisfy $f(H) = H'$. However, for any EALA $(E, H)$ and
automorphism $f$ of $E$ it is easily seen that $\big(E,f(H)\big)$ is an EALA
which satisfies $\big(f(H))_{cc} = f_{cc}(H_{cc})$. We can therefore split a
solution of problem (B) into two steps:

\begin{enumerate}   \item[(B1)] (\cite[Thm.~7.1]{CNPY}) If $H_{cc} =
H'_{cc}$ then there exists $f\in \Aut_k(E)$ such that $f(H) = H'$.

\item[(B2)] The automorphism used to solve problem (A) can be lifted to an
    automorphism of $E$.
\end{enumerate}
\sm

We have solved Problem (B2) and thus established the Conjugacy Theorem for
extended affine Lie algebras in the fgc case in \cite[Thm.~7.6]{CNPY}.  Thus
the Conjugacy Theorem for extended affine Lie algebras is reduced to proving
(B2) in the non-fgc case. \sm

As explained in \ref{lietor-prop}\eqref{lietor-prop-d}, 
in the non-fgc case  $E_{cc} \simeq \lsl_\ell(\scQ)$ for some  $\ell \ge 2$,
and $\scQ$ a quantum torus which is not finitely generated over its centre. But
as in \cite{CNP} we will deal here with the Lie algebra $ \lsl_\ell(\scQ)$ for
an {\it arbitrary} quantum torus $\scQ.$\footnote{Assuming that $Q$ is not-fgc
would not simplify our arguments. The additional generality may be of future
independent interest.}  The conjugacy theorem of \cite{CNP} for
$L=\lsl_\ell(\scQ)$, i.e., the solution of Problem (A) in the non-fgc case,
uses an interior automorphism $\Int(g)$ for some $g\in \GL_\ell(\scQ)$. The
final step in the proof of the Conjugacy Theorem for EALAs is therefore that
such $g$ can be suitably chosen. More precisely. \sm

{\bf Main Theorem.} {\em Let $L=\lsl_\ell(\scQ)$, $\ell\ge 2$ with $\scQ$ a
quantum torus, then Problem\/ {\rm (A)}  can be solved with a $g\in
\GL_\ell(\scQ)$ such that  $\Int(g)$ can be lifted to an automorphism of any
extended affine Lie algebra $E$ with $E_{cc} = L$.} \sm

The somewhat curious formulation of our result refers to the fact that we are
not claiming that all automorphisms $\Int(g)$ can be lifted to the EALA
level.\sm

\begin{remark}\label{fieldnature} A word on the nature of our base field $k$.
 The solution of Problem (A) in the fgc case (\cite{CGP}) assumes $k$ algebraically closed
(and of course of characteristic $0$). The reason for this assumption is the
Realization Theorem of  \cite{abfp2}. More precisely, \cite{CGP} holds as long
as one knows that $E_{cc}$ is a multiloop algebra, while \cite{abfp2} shows
that this holds in the fgc case under the assumption that $k$ be algebraically
closed.

In the non-fgc case $k$ there is no need not to assume that $k$ be algebraically closed to solve
problem (A)  (see \cite{CNP}). The lifting result (B1) works for
any field of characteristic $0$. In the remainder of this paper we will assume
that our base field $k$ has characteristic $0$, but need not be algebraically
closed. It is in this setting that we will prove our Main Theorem in the non-fgc case, namely the Conjugacy Theorem for EALAs with a non-fgc centreless core.
\end{remark}
\bigskip

{\em Notation.} For elements $g,h$ of a group $G$ we denote by $\dbl g,h \dbr =
ghg^{-1} h^{-1}$ the commutator of $g$ and $h$, and by $\scD(G)= (G,G)$ the
commutator subgroup of $G$. As usual $\Int(g)(h) = g h g\mn$. We use $D<L$ to
indicate that $D$ is a subalgebra of the algebra $L$. For any (associative or
Lie) algebra $A$ we denote by $\Der_k(A)$ the Lie algebra of $k$-linear
derivations of $A$, and by $\scZ(A)$ its centre.
\bigskip

\section{Interlaced extensions} \label{sec:rev}

 In this section we introduce a general construction of Lie algebras,
so-called interlaced extensions. We will see in \S\ref{sec:review} that
extended affine Lie algebras are examples of interlaced extensions. In addition,
one of the principal components of our proof of the Main Theorem can and will
be done in the setting of interlaced extensions (Theorem~\ref{etap}).

\subsection{Cocycles}\label{cohom} Let $L$ be a Lie algebra and let $V$ be an
$L$-module. A {\em $2$-cocycle with coefficients in $V$} is an  alternating
map $\si\co L \times L \to V$ satisfying for $l_i \in L$
\begin{equation}\label{cohomo1}
\begin{split}
  & l_1 \cdot \si(l_2, l_3) + l_2 \cdot \si(l_3, l_1) + l_3 \cdot \si(l_1, l_2)
  \\ &\qquad = \si([l_1, l_2], l_3) + \si([l_2, l_3], l_1) + \si([l_3, l_1], l_2) \end{split}  \end{equation}
Given such a $2$-cocycle $\si$, the vector space $L \oplus V$ becomes a Lie
algebra with respect to the product
\[  [l_1 + v_1, \, l_2  + v_2] = [l_1,
l_2]_L + \big( l_1 \cdot v_2 - l_2 \cdot v_1 + \si(l_1, l_2) \big).
\]
We will denote this Lie algebra by $L \oplus_\si V$. Note that the projection
onto the first factor $\pr_L \co L \oplus_\si V \to L$ is an epimorphism of
Lie algebras whose kernel is the abelian ideal $V$. We refer to such an
extension as an {\em abelian extension}. \sm

A special case of this construction is the situation when $V$ is a trivial
$L$-module. In this case a $2$-cocycle will be called a {\em central
$2$-cocycle\/}. Note that all terms on the left hand side of \eqref{cohomo1}
vanish. For a central $2$-cocycle, $\pr_L \co L \oplus_\si V \to L$ is an
epimorphism whose kernel $V$ is the central ideal $V$ of $L \oplus_\si V$,
i.e., $\pr_L$ is a {\em central extension}. \sm

A basic construction of a central $2$-cocycle goes as follows. We assume that
$\be$ is a bilinear form on $L$ which is {\em invariant\/} in the sense that
$\be([l_1, l_2], l_3) = \be(l_1, [l_2, l_3])$ holds for all $l_i \in L$. We
denote by
\[
   \SDer_{k,\beta}(L)
\]
(or simply $\SDer_{k}(L)$ if $\beta$ is fixed within our context) the subalgebra of $\Der_{k}(L)$ consisting of {\em skew derivations}, i.e.,
those derivations $d$ satisfying $\be \big(d (l) , l\big) = 0$ for all $l\in
L$. We further suppose that $D$ is a Lie algebra acting on $L$ by skew
derivations.
It is well-known and easy to check that
\begin{equation}\label{eala-cons0}
\si_{D,\beta} \co L \times L   \to D^*:= \Hom_k(D, k),
\quad \si_{D,\be}\big( l_1, \, l_2)\, (d) = \be \big( d(l_1) , \, l_2)\end{equation} is a central
$2$-cocycle.

\subsection{Interlaced extensions.}\label{gen-data}
As we explained in the introduction, one of the main problems solved in this
paper is lifting an automorphism from the centreless core $\lsl_\ell(Q)$ of an
EALA $E$ to $E$. We will see that this can be done without additional work
 in a more general setting than extended affine Lie algebras. By working
on this more general edifice not only do we strip the lifting process from
unnecessary assumptions, but  we also  suggest the possibility of
recasting EALA theory in a more general cadre. In this subsection we will
introduce this general framework. It uses the following data:\sm

\begin{enumerate}[(i)]
  \item\label{gen-datai} a Lie algebra $L$ equipped with an invariant
      bilinear form $\beta$; \sm

 \item \label{gen-dataii} a Lie algebra $D$ acting on $L$ by skew
     derivations of $(L, \beta)$; we write this action as $d\cdot l$ or
     sometimes $d(l)$ for $d\in D$ and $l\in L$; \sm

  \item \label{gen-dataiii} a subspace $C\subset D^*$ which is invariant
      under the co-adjoint action of $D$ on $D^*$, defined by $(d\cdot
      c)(d') = c([d', d])$, and satisfies
\begin{equation}
  \label{gen-dataiii1} \si_{D, \be} (l_1, l_2) \in C \quad (l_i \in L)
\end{equation}
for $\si_{D, \be}$ as in \eqref{eala-cons0}; \sm

 \item\label{gen-dataiv} a $2$-cocycle $\ta \co D \times D \to C$.
\end{enumerate}
\sm

\noindent Given these data, we define a product on the vector space
\begin{equation} \label{gen-data1}
 E= L \oplus C \oplus D
\end{equation}
by ($l_i \in L$, $c_i \in C$ and $d_i \in D$)
\begin{equation} \label{n:gencons3}
\begin{split}
  & [ l_1 + c_1 + d_1, \, l_2 + c_2 + d_2 ]
   = \big( [l_1, l_2]_L + d_1 \cdot l_2 - d_2 \cdot l_1 \big)
   \\&\qquad  \oplus
     \big(\si_{D,\be}(l_1, l_2) +d_1 \cdot c_2 - d_2 \cdot c_1
      + \ta(d_1, d_2) \big)
    \oplus  [d_1, d_2]_D.
 \end{split} \end{equation}
In this formula $[.,.]_L$ and $[.,.]_D$ denote the Lie algebra products of
$L$ and $D$ respectively. We use $\oplus$ on the right hand side of
\eqref{n:gencons3} as a mnemonic device to indicate the components of the
product with respect to the decomposition \eqref{gen-data1}. To avoid any
possible confusion we will sometimes indicate the product of $E$ by
$[.,.]_E$. We often abbreviate $\si = \si_{D,\be}$. \sm

Our construction is a special case of \cite[1.4]{CNPY}. Thus, by
\cite[1.5]{CNPY}, the vector space $E$ together with the product
\eqref{n:gencons3} is a Lie algebra. Since it is obtained by interlacing the
central extension $ 0 \to C \to  L \oplus C \to L \to 0$ (obvious maps) with
the abelian extension $0 \to C \to C \oplus D \to D \to 0$ (again obvious maps)
we call this Lie algebra the {\em interlaced extension given by the data
$(L,\be, D,C, \tau)$\/} and denote it $\IE(L,D,C)$ or $\IE(L,\be, D,C, \tau)$
if more precision is helpful.

Later on the bilinear form $\be$ on $L$ will be unique, up to a scalar in
$k^\times$. In general, we have for $s\in k^\times$
\begin{equation}
  \si_{D, s\be} = s\si_{D,\be} \quad\hbox{and} \quad
    \IE(L, \be, D, C, \ta) \simeq \IE(L, s\be, D, C, s\ta)
\end{equation}
via the isomorphism $l \oplus c \oplus d \mapsto l \oplus sc \oplus d$.

\begin{Lemma}\label{autchar}\footnote{ This lemma holds in the more general
setting of \cite[1.4]{CNPY}. But we have no use for this generality.} Let $E=
\IE(L,D,C)= L \oplus C \oplus D$ be an interlaced extension, and let $f \co E
\to E$ be a linear map of the form
\begin{equation} \label{autchar0}
 f(l \oplus c \oplus  d) = \big(f_L(l) + \eta(d) \big) \oplus
     \big( \psi(l) + c + \vphi(d) \big) \oplus d
\end{equation}
where $l\in L$, $c\in C$, $d\in D$ and
\begin{equation} \label{autchar00}
 f_L \co L \to L, \quad \eta\co D \to L, \quad \psi \co L \to C, \quad
  \vphi \co D \to C
\end{equation}
are linear maps. Then $f$ is an automorphism of the Lie algebra $E$ if and
only if the following conditions hold for all $l,l_1, l_2 \in L$ and $d,d_1,
d_2 \in D$. \ms
\begin{enumerate}[\rm (a)]
 \item\label{autochar1} $f_L$ is an automorphism of the Lie algebra $L$,
     \ms

\item \label{autochar2} $\si\big(f_L(l_1), f_L(l_2)\big) = \psi\big( [l_1,
    l_2]\big) + \si(l_1, l_2)$ for $\si = \si_{D,\be}$, \ms

\item \label{autochar3} $f_L( d\cdot l)  = [\eta(d), f_L(l)]_L + d\cdot
    f_L(l)$, \ms

\item \label{autochar4} $\psi(d\cdot l) = \si\big( \eta(d), f_L(l)\big) +
    d\cdot \psi(l)$, \ms

\item \label{autochar5} $\eta\big( [d_1, d_2]_D\big) = [ \eta(d_1), \,
    \eta(d_2) ]_L + d_1 \cdot  \eta(d_2) - d_2 \cdot \eta(d_1)$, \ms

\item \label{autochar6} $\vphi\big( [d_1, d_2]\big) = \si\big( \eta(d_1),
    \, \eta(d_2) \big) + d_1 \cdot \vphi(d_2) - d_2 \cdot \vphi(d_1)$.

\end{enumerate}
\end{Lemma}

\begin{proof}
The map $f$ is bijective if and only if $f_L$ is so. Moreover, the definition
of the product of $E$ in \eqref{n:gencons3} and the definition of $f$ in
\eqref{autchar0} show that $f$ is a homomorphism of the Lie algebra $E$ if
and only if it respects the products $[l_1, l_2]_E$, $[d,l]_E$ and $[d_1,
d_2]_E$. This leads to the conditions \eqref{autochar1}--\eqref{autochar6}.
\end{proof}
\sm

We will call an automorphism of type \eqref{autchar0} a {\em special
automorphism\/}.  Not all automorphisms of $E$ are special, but we have the following result.

\begin{proposition} \label{elemlift}
{\rm (\cite[Prop.~1.6]{CNPY})} Let $E=\IE(L,D,C)$ be an interlaced extension.
Every elementary automorphism of $L$ lifts to a special automorphism of $E$.
\end{proposition}

We recall that an elementary automorphism of a Lie algebra $M$ is a product of
automorphisms of type $\exp \ad_M x$ with $\ad_M x \in \End_k(M)$ (locally)
 nilpotent.  The reader can easily verify that for $f_L = \exp \ad_L x$, the maps $\eta$,
$\psi$ and $\vphi$ of \eqref{autchar00} are given by
\begin{align*}
  \psi(l) &= \textstyle \sum_{n\ge 1} \frac{1}{n!} \, \si\big(x, (\ad_L x)^{n-1}(l)\big)
     \quad   \hbox{for $\si = \si_{D,\be}$} , \\
 \eta(d) &= - \textstyle \sum_{n\ge 1} \frac{1}{n!} (\ad_L x)^{n-1} (d\cdot x), \\
 \vphi(d) &= - \textstyle \sum_{n\ge 2} \frac{1}{n!} \, \si\big(x, \, (\ad_L x)^{n-2}
                   (d\cdot x) \big)
 \end{align*}
These formulas indicate that the maps $\psi$, $\eta$ and $\vphi$ are in
general not zero.

\subsection{Enlarging interlaced extensions.}\label{enl} In the process of lifting
an automorphism from $L$ to an interlaced extension $E$, we will enlarge $E$
to a bigger interlaced extension using the following construction.
\sm
\begin{enumerate}[(i)]
   \item $E=\IE(L, \be, D,C, \tau)$ is an interlaced extension; \sm

  \item \label{enl-ii} $L$ is a subalgebra of a Lie algebra $L'$ equipped
      with an invariant bilinear form $\be'$ such that $\be' \mid_{L\times
      L} = \be$; \sm

 \item \label{enl-iii} the action of $D$ on $L$ extends to an action of $D$
     on $L'$ by skew derivations, and \sm

 \item \label{enl-iv} $\si_{D,\be'}(l'_1, l'_2) \in C$ for $l_1', l_2'\in
     L$.\footnote{ Note that because of assumption (iii) we necessarily have
     that $\si_{D, \be'} \co L' \times L'\to C \subset D^*$ coincides with
     the central $2$-cocycle of \eqref{eala-cons0} when restricted to $L
     \times L$.}
  \end{enumerate}
\sm

 The data $(L',\be', D, C, \ta)$ satisfy the assumptions
\eqref{gen-datai}--\eqref{gen-dataiv} of \ref{gen-data}, so that we can form
the interlaced extension
\[ E' = \IE(L',\be', D, C, \ta) = L'\oplus C \oplus D
\]
Since for $l_1, l_2 \in L$ we have $\si_{D,\be'}(l_1, l_2)(d) = \be'(d \cdot
l_1, l_2) = \be(d\cdot l_1, l_2) = \si_{D, \be}(l_1, l_2)(d)$, i.e.,
\[ \si_{D, \be'} \big |_{L \times L} = \si_{D, \be},
\]
it is immediate that {\em $E$ is a subalgebra of $E'$}. \ms

In this setting suppose that $f'$ is a special automorphism of $E'$, thus
given by the data
\[
 f'_{L'} \co L' \to L', \quad \eta'\co D \to L', \quad \psi' \co L' \to C, \quad
  \vphi' \co D \to C
\]
as in \eqref{autchar00}, satisfying the conditions
\eqref{autochar1}--\eqref{autochar6} of Lemma \ref{autchar}. It is then immediate
that
\begin{equation} \label{enl2}
 \hbox{\em $f'(E) = E \quad \iff \quad f'_{L'}(L) = L$ and $\eta'(D) \subset L$.}
\end{equation}
In this case $f'\mid_E$ is obviously an automorphism of $E$, in fact a
special automorphism given by the data
\begin{equation} \label{enl3}
 f_L = f'_{L'} \big | _{L}, \quad \eta = \eta', \quad \psi=\psi', \quad \vphi = \vphi'.
\end{equation}

\section{Review: Lie tori and extended affine Lie algebras (EALAs)} \label{sec:review}

In this section we review the theory of extended affine Lie algebras, in order
to give the reader a perspective about the achievements of this paper. The
structure of extended affine Lie algebra is intimately connected to Lie tori.
We therefore start with a short summary of the pertinent facts from the theory
of Lie tori. We then introduce EALAs and describe their construction as a
special case of an interlaced extension (\ref{gen-data}) based on a Lie torus.

\subsection{Lie tori} \label{def:lietor} We use the term ``root system'' to mean
a finite, not necessarily reduced root system $\Delta$ in the usual sense,
except that we will assume $0 \in \Delta$, as for example in \cite{AABGP}. We
denote by $\Delta\ind = \{ 0 \} \cup \{ \alpha\in \Delta: \frac{1}{2} \alpha
\not\in \Delta\}$ the subsystem of indivisible roots and by
$\euQ(\Delta)=\Span_\ZZ(\Delta)$ the root lattice of $\Delta$. To avoid some
degeneracies we will always assume that $\Delta\ne \{0\}$.

\smallskip Let $\Delta$ be a finite irreducible root system, and let $\La$ be
free abelian group of finite type.\footnote{ Thus $\La\cong \ZZ^n$ for some
$n\in \NN$. But it is not helpful to assume $\La=\ZZ^n$.} A \textit{Lie torus
of type $(\Delta,\La)$\/} is a Lie algebra $L$ satisfying the following
conditions (LT1) -- (LT4).
\smallskip
\begin{itemize}

\item[(LT1)] (a) $L$ is graded by $\euQ(\Delta) \oplus \La$. We write this
    grading as $L = \bigoplus_{\alpha \in \euQ(\Delta), \la \in \La}
    L_\alpha^\la$ and thus have $[L_\alpha^\la, L_\beta^\mu] \subset L^{\la +
    \mu}_{\alpha + \beta}$. It is convenient to define \begin{equation}
    \label{def:lietor1}
     L_\alpha = \ts \bigoplus_{\la \in \La} L_\alpha^\la \quad
    \hbox{and}\quad L^\la = \bigoplus_{\alpha \in \euQ(\Delta)}
    L_\alpha^\la.\end{equation}
     (b) We further assume that $\supp_{\euQ(\Delta)} L = \{ \alpha \in
    \euQ(\Delta); L_\alpha \ne 0\} = \Delta$, so that $L = \bigoplus_{\alpha
    \in \Delta} L_\alpha$.

\item[(LT2)] (a) If $L_\alpha^\la \ne 0$ and $\alpha \ne 0$, then there exist
    $e_\alpha^\la \in L_\alpha^\la$ and $f_\alpha^\la \in L_{-\alpha}^{-\la}$
    such
    that  
     \[  L_\alpha^\la = k e_\alpha^\la, \quad L_{-\alpha}^{-\la} = k f_\alpha^\la,
                \]
and
 \begin{equation*} \label{def:lietor2}
   [[e_\alpha^\la, f_\alpha^\la],\, x_\beta] = \lan \beta, \alpha\ch\ran x_\beta
 \end{equation*}
for all $\beta \in \Delta$ and $x_\beta \in L_\beta .$\footnote{ Here and
elsewhere $\alpha\ch$ denotes the coroot corresponding to $\alpha$ in the
 sense of \cite{Bbk}.}

(b) $L_\alpha^0 \ne 0$ for all $0 \ne \alpha \in \Delta\ind.$ \smallskip

\item[(LT3)]  As a Lie algebra, $L$ is generated by $\bigcup_{0\ne \alpha \in
    \Delta} L_\alpha$. \smallskip

\item[(LT4)] As an abelian group, $\La$ is generated by $\supp_\La L = \{ \la
    \in \La : L^\la \ne 0\}$.
\end{itemize} \smallskip

We define the {\em nullity} of a Lie torus $L$ of type $(\Delta,\La)$ as the
rank of $\La$. 
We will say that $L$ is a \emph{Lie torus} (without qualifiers) if $L$ is a Lie
torus of type $(\Delta,\La)$ for some pair $(\Delta,\La)$. A Lie torus is
called \textit{centreless\/} if its centre $\scZ(L) = \{0\}$. If $L$ is an
arbitrary  Lie torus, its centre $\scZ(L)$ is contained in $L_0$ from which it
easily follows that $L/\scZ(L)$ is in a natural way a centreless Lie torus of
the same type as $L$ and nullity (see \cite[Lemma~1.4]{y:lie}). \sm

The structure of Lie tori is known, see \cite{Al} for a recent survey. Some
more background on Lie tori is given in the papers \cite{abfp2,n:persp,
n:eala-summ}. Lie tori can of course be defined for any abelian group $\La$ (see
for example \cite{y:lie}), but only the case of a free abelian group of finite
rank is of interest for EALAs. \sm

An obvious example of a Lie torus of type $(\Delta,\ZZ^n)$ is the Lie
$k$-algebra $\g \ot R$ where $\g$ is a finite-dimensional split simple Lie
algebra of type $\Delta$ and $R=k[t_1^{\pm 1}, \ldots, t_n^{\pm 1}]$ is the
Laurent polynomial ring in $n$-variables with coefficients in $k$ equipped with
the natural $\ZZ^n$-grading. Another important example, first studied in
\cite{bgk}, is the Lie algebra $\lsl_\ell(\scQ)$ for $\scQ$ a quantum torus
(see \ref{qua-pro} and \ref{link}).

\subsection{Some known properties of centreless Lie tori}\label{lietor-prop}
We review some of the properties of Lie tori needed in the following.
We assume that $L$ is a centreless Lie torus of type $(\Delta,\La)$ and nullity
$n$.
\smallskip

\begin{inparaenum}[(a)]
\item \label{lietor-prop-a} For $e^\la_\alpha$ and $f_\alpha^\la$ as in (LT2)
    we     put    $h_\alpha^\la    =    [e_\alpha^\la, f_\alpha^\la]\in L_0^0$ and observe that $(e_\alpha^\la,
    h_\alpha^\la, f_\alpha^\la)$ is an $\lsl_2$-triple. Then \begin{equation}
    \label{def:h}
   \frh = \Span_k \{ h^\la_\alpha\} = L_0^0\end{equation}   is a toral
\footnote{A subalgebra $T$ of a Lie algebra $L$ is toral, sometimes also called
$\ad$-diagonalizable, if $L = \bigoplus_{\al\in T^*} L_\al(T)$ for $L_\al(T) =
\{ l \in L : [t,l] = \al(t)l \hbox{ for all $t\in T$}\}$. In this case $\{ \ad
t : t\in T\}$ is a commuting family of $\ad$-diagonalizable endomorphisms.
Conversely, if $\{ \ad t : t\in T\}$ is a commuting family of
$\ad$-diagonalizable endomorphisms and $T$ is a finite-dimensional subalgebra,
then $T$ is a toral.} subalgebra of $L$ whose root spaces are the $L_\alpha$,
$\alpha \in \Delta$. \sm

\item\label{lietor-prop-b} Up to scalars, $L$ has a unique nondegenerate
 symmetric bilinear form $\inpr$ which is {\em $\La$-graded\/} in the sense that $(L^\la \mid L^\mu) = 0$ if
    $\la + \mu \ne 0$, \cite{NPPS, y:lie}. Since the subspaces $L_\alpha$ are
    the root spaces of the toral subalgebra $\frh$ we also know $(L_\alpha \mid
    L_\ta) = 0$ if $\alpha + \ta \ne 0$. \smallskip

\item\label{lietor-prop-c} Let $\Ctd_k(L) = \{ \chi \in \End_k(L) :
    \chi([l_1,     l_2]) = [l_1,    \chi(l_2)] \; \forall \, l_1, l_2\in L \}$ be the centroid of $L$ (see for
    example \cite{bn} for general facts about centroids). Since $L$ is perfect,
    $\Ctd_k(L)$ is a commutative associative unital subalgebra of $\End_k(L)$.
    It is graded with respect to the $\La$-grading \eqref{def:lietor1} of $L$:
\[
    \Ctd(L) = \textstyle \bigoplus_{\la \in \La}\,  \Ctd(L)^\la
\] where
$\Ctd(L)^\la$ consists of those centroidal transformations $\chi$ satisfying
$\chi(L^\mu) \subset L^{\la + \mu}$ for all $\mu \in \La$. One knows that
$\Ctd_k(L)$ is graded-isomorphic to the group ring $k[\Xi]$ for a subgroup
$\Xi$ of $\La$, the so-called {\em central grading group}. Hence $\Ctd_k(L)$ is
a Laurent polynomial ring in $\nu$ variables, $0 \le \nu \le n$,
(\cite[7]{n:tori}, \cite[Prop.~3.13]{bn}). (All possibilities for $\nu$ do in
fact occur, for example for $L=\lsl_\ell(\scQ)$, see \ref{qua-pro} and
\ref{link}.) \sm

\item \label{lietor-prop-d} The space $L$ is naturally a $\Ctd_k(L)$-module via
    $\chi    \cdot    a  = \chi(a)$. As a $\Ctd_k(L)$-module, $L$ is free.
    If $L$ is {\em fgc}, i.e.,  namely finitely generated as a module over its centroid, then $L$ is a
    multiloop algebra \cite{abfp2}. If $L$ is not fgc, equivalently $\nu < n$,
    one knows (\cite[Th.~7]{n:tori}) that $L$ has root-grading type ${\rm A}$.
    Lie tori with this root-grading type are classified in \cite{bgk,bgkn,y1}.
    It follows from this classification together with \cite[4.9]{ny} that
    $L\simeq \lsl_l(\scQ)$ for $\scQ$ a quantum torus in $n$ variables and
    structure matrix $q=(q_{ij})$ an $n\times n$ quantum matrix with at least
    one $q_{ij}$ not a root of unity (\ref{qua-pro}).
\smallskip

\item \label{lietor-prop-e} Any $\theta \in \Hom_\ZZ(\La, k)$ induces a
    so-called {\em degree
    derivation} $\pa_\theta$ of $L$ defined by $\pa_\theta (l^\la) =
    \theta(\la) l^\la$ for $l^\la \in L^\la$. We put $\euD = \{ \pa_\theta:
    \theta \in \Hom_\ZZ(\La, k) \}$ and note that $\theta \mapsto \pa_\theta$
    is a vector space isomorphism from $\Hom_\ZZ(\La, k)$ to $\euD$, whence
    $\euD\simeq k^n$. We define $\ev_\la \in \euD^*$ by $\ev_\la(\pa_\theta) =
    \theta(\la)$. One knows (\cite[8]{n:tori}) that $\euD$ induces the
    $\La$-grading of $L$ in the sense that $L^\la= \{ l \in L : \pa_\theta(l) =
    \ev_\la(\pa_\theta) l \hbox{ for all } \theta \in \Hom_\ZZ(\La, k)\}$ holds
    for all $\la \in \La$.\smallskip

\item If $\chi \in \Ctd_k(L)$ then $\chi d \in \Der_k(L)$ for any derivation
    $d\in \Der_k(L)$. We call
\begin{equation}\label{Volodya1}  \CDer_k(L) := \Ctd_k(L) \euD
   = \ts \bigoplus_{\xi \in \Xi} \, \Ctd(L)^\xi \euD\end{equation}
the {\em centroidal derivations\/} of $L$.
It is easily seen that $\CDer(L)$ is a $\Xi$-graded subalgebra of $\Der_k(L)$,
a  generalized Witt algebra. Note that $\euD$ is a toral subalgebra of $\CDer_k(L)$
whose root spaces are the $\Ctd(L)^\xi \euD = \{ d\in \CDer(L): [t, d] = \ev_\xi(t) d \hbox{ for all }
t\in \euD\}$. One also knows (\cite[9]{n:tori}) that
\begin{equation} \label{lietor-prop1}
\Der_k(L) = \IDer(L) \rtimes \CDer_k(L) \quad (\hbox{semidirect product}),
\end{equation}
where $\IDer(L)$ is the ideal of inner derivations of $L$.
\sm

\item \label{lietor-prop-g} For the construction of EALAs, the $\Xi$-graded
    subalgebra     $\SCDer_k(L)$    of {\em skew-centroidal derivations\/} is important:
\begin{align*}
    \SCDer_k(L) &= \{ d\in \CDer_k(L) : (d \cdot l \mid l) = 0 \hbox{ for all } l\in L\}
  \\ &=  \ts \bigoplus_{\xi \in \Xi } \SCDer_k(L)^\xi,  \\
   \SCDer_k(L)^\xi &=  \Ctd(L)^\xi \{ \pa_\theta : \theta(\xi) = 0 \}.
\end{align*}
\end{inparaenum}

\subsection{Extended affine Lie algebras (EALAs)} \label{def:eala}
 An \textit{extended affine Lie algebra\/} or EALA
for short, is a pair $(E,H)$ consisting of a Lie algebra $E$ over $k$ and a
subalgebra $H$ of $E$
satisfying the axioms (EA0) -- (EA5) below.
\begin{itemize}
\item[(EA0)] $E$ has an invariant nondegenerate symmetric bilinear form
    $\inpr$. \smallskip

\item[(EA1)] $H$ is a nontrivial finite-dimensional toral  and
    self-centra\-li\-zing subalgebra of $E$.
\end{itemize}
Thus $E = \ts\bigoplus_{\al \in H^*} E_\al $ for $E_\al = \{ e\in E: [h,e] =
\al(h)e \hbox{ for all } h\in H\}$ and $E_0 = H$. We denote
     by $\Psi=\{\al \in H^*: E_\al \ne 0\}$ the set of roots of $(E,H)$  --
     note that $0 \in \Psi$! Because the restriction of $\inpr$ to $ H \times H $ is nondegenerate,
     one can in the usual way transfer this bilinear form to $H^*$ and
     then introduce anisotropic roots $\Psi\an= \{ \al \in \Psi : (\al \mid
     \al) \ne 0\}$ and isotropic (= null) roots $\Psi^0 = \{ \al\in \Psi : (\al
     \mid \al) = 0\}$. The \textit{core of $\big(E,H, \inpr \big)$} is by definition the
     subalgebra generated by $\bigcup_{\al \in \Psi\an} E_\al$. It will be henceforth denoted by $E_c.$\smallskip
\begin{itemize}
\item[(EA2)] For every $\al \in \Psi\an$ and $x_\al \in E_\al$, the operator
    $\ad x_\al$ is locally nilpotent on $E$. \smallskip

\item[(EA3)] $\Psi\an$ is connected in the sense that for any decomposition
    $\Psi\an = \Psi_1 \cup \Psi_2$ with $\Psi_1 \neq \emptyset$ and $\Psi_2
    \neq \emptyset$ we have $(\Psi_1 \mid \Psi_2) \neq 0$.
   \smallskip

\item[(EA4)] The centralizer of the {\em core $E_c$\/} of $E$ is contained in
    $E_c$, i.e., $\{e \in E : [e, E_c] =0 \} \subset E_c$. \smallskip

\item[(EA5)] The subgroup $\La = \Span_\ZZ(\Psi^0) \subset H^*$ generated by
    $\Psi^0$ in $(H^*,+)$ is a free abelian group of finite rank.
\end{itemize}

The attentive reader will have noticed that the choice of invariant
nondegenerate symmetric bilinear form in (EA0) is part of the structural data
defining an EALA. However, one can show that another choice of an invariant
nondegenerate symmetric bilinear form leads to the same set of anisotropic and
isotropic roots $\Psi\an$ and $\Psi^0$, and thus also to the same core $E_c$
and {\em centreless core $E_{cc} = E_c/\scZ(E_c)$\/}, see \cite[Rem.~2.4 and
Cor.~3.3]{CNPY}. The core $E_c$ of an EALA $E$ is always an ideal of $E$. \sm

Some references for EALAs are \cite{AABGP, bgk, Ne4, n:persp, n:eala-summ}. It
is immediate that any finite-dimensional split simple Lie algebra $\g$ is an
EALA of nullity $0$ and $\g=\g_c=\g_{cc}$. The converse is also true,
\cite[Prop.~5.3.24]{n:eala-summ}. It is also known that any affine Kac-Moody
algebra over $\bbC$ is an EALA -- in fact, by \cite{abgp}, the affine Kac-Moody
algebras $\g$ are precisely the EALAs over $\bbC$ of nullity $1$. For those,
$\g_c=[\g, \g]$ and $\g_{cc}$ is an (twisted or untwisted) loop algebra.

\subsection{The roots of an EALA}\label{rev:root} The set $\Psi$ of  roots
of an EALA $(E,H)$ is an extended affine root system in the sense of
\cite[Ch.~I]{AABGP} (see also the surveys \cite[\S2, \S3]{n:persp} and
\cite[\S5.3]{n:eala-summ}). Thus, there exists an irreducible finite (but
possibly non-reduced) root system $\Delta \subset H^*$, an embedding
$\Delta\ind \subset \Psi$ and a family $(\La_\alpha : \alpha \in \Delta)$ of
subsets $\La_\alpha \subset \La$ such that
\begin{equation} \label{root1}
  \Span_k(\Psi) = \Span_k(\Delta) \oplus \Span_k(\La) \quad \hbox{and} \quad
  \Psi = \ts \bigcup_{\alpha \in \Delta} ( \alpha + \La_\alpha).
\end{equation}
Using this (non-unique) decomposition of $\Psi$, we write any $\psi \in \Psi$
as $\psi = \al + \la$ with $\al \in \Delta$ and $\la \in \La_\alpha \subset
\La$ and define $(E_c)_\al^\la = E_c \cap E_\psi$. Then the core $E_c =
\bigoplus_{\alpha \in \Delta, \la \in \La} (E_c)_\alpha^\la$ is a Lie torus of
type $(\Delta,\La)$, and the centreless core $E_{cc} = E_c/\scZ(E_c)$ is a
centreless Lie torus.

\subsection{Construction of EALAs} \label{eala-cons} To construct an EALA one reverses the
process described in \ref{rev:root}. We will use data $(L,\sigma_D,\ta)$
described below. Some more background material can be found in
\cite[\S6]{n:persp} and \cite[\S5.5]{n:eala-summ}:
\begin{itemize}

\item $L$ is a centreless Lie torus of type $(\Delta,\La)$. We fix a
    $\La$-graded invariant nondegenerate symmetric bilinear form $\inpr$ (see
    \ref{lietor-prop}\eqref{lietor-prop-b}) and let $\Xi$ be the central
    grading group of $L$ (see \ref{lietor-prop}\eqref{lietor-prop-c}). \sm

 \item $D=\bigoplus_{\xi \in \Xi} D^\xi$ is a graded subalgebra of $
     \SCDer_k(L)$ (see \ref{lietor-prop}\eqref{lietor-prop-g}) such that the
     evaluation map $\ev_{D^0} : \La \to D^{0\, *}$, $\la \to \ev_\la
     \mid_{D^0}$, defined in \ref{lietor-prop}\eqref{lietor-prop-e},  is
     injective.

     Since $(L^\la \mid L^\mu) =0$ if $\la + \mu \ne 0$ and since
     $D^\xi(L^\la) \subset L^{\xi + \la}$ it follows that the central cocycle
     $\si_D$ of \eqref{eala-cons0} has values in the graded dual \[ D^{\gr
     *}=:C\]  of $D$. Recall $C=\bigoplus_{\xi \in \Xi} C^\xi$ with $C^\xi =
     (D^{-\xi})^* \subset D^*$. The contragredient action of $D$ on $D^*$
     leaves $C$ invariant.
     \sm

  \item $\ta : D\times D \to C$ is an \textit{affine cocycle\/} defined to be
      a $2$-cocycle   satisfying $\ta(d^0, d) = 0$ and $\ta(d_1, d_2)(d_3) =
      \ta(d_2, d_3)(d_1)$ for all $d,d_i \in D$ and $d^0 \in D^0$.

     \noindent It is important to point out that there do exist non-trivial
affine cocycles, see \cite[Rem.~3.71]{bgk}.
 \end{itemize}
\sm

\noindent The data $(L,\sigma_D,\ta)$ with $\be$ the unique invariant bilinear
form $\inpr$ of \ref{lietor-prop}\eqref{lietor-prop-b} satisfy all the axioms
of our general construction \ref{gen-data}. Thus the interlaced extension
\begin{equation}\label{eala-cons1} E = L \oplus C \oplus D \end{equation}
is a Lie algebra with respect to the product \eqref{n:gencons3}. Note that $E$
contains the toral subalgebra
\[H =\frh \oplus C^0\oplus D^0\]
for $\frh$ as in \eqref{def:h}. The symmetric bilinear form $\inpr$
on $E$, defined by
\[ \big( l_1 +  c_1 + d_1 \mid l_2 + c_2 + d_2\big)
  = (l_1 \mid l_2)_L + c_1(d_2) + c_2(d_1), \]
is nondegenerate and invariant, thus fulfilling the axiom (E0). \sm

\textbf{Examples:} (a) In case $L=\g$ is a finite-dimensional split simple Lie
algebra, $\Ctd(\g) = 0$, $\Xi=\{0\}$, and so also $\SCDer(\g)=0$. The
construction above therefore yields $E=\g$.

(b) In case $L$ is a twisted or untwisted loop algebra based on $\g$ as in (a)
over $\bbC$, the centroid $\Ctd(L)$ is isomorphic to a Laurent polynomial ring
$R$, $\CDer(L)$ is a free $R$-module of rank $1$, but $\SCDer(\g)$ is
$1$-dimensional over $\bbC$. The only non-trivial choice is therefore $D\simeq
\bbC$. In this case necessarily $\tau = 0$.  Thus the construction of affine
Kac-Moody algebras is a special case of our construction above.

\begin{theorem}[{\cite[Th.~6]{Ne4}}]\label{n:mainconst} {\rm (a)} The triple
$\big(E,H, \inpr \big)$ constructed above is an extended affine Lie algebra,
denoted $\EA(L,D,\ta)$. Its core is $L \oplus D^{\gr\, *}$ and its centreless
core is $L$. \smallskip

{\rm (b)} Conversely, let $\big(E,H, \inpr \big)$ be an extended affine Lie
algebra, and let $L=E_c/\scZ(E_c)$ be its centreless core. Then there exists a
subalgebra $D\subset \SCDer_k(L)$ and an affine cocycle $\ta$ satisfying the
conditions in {\rm \ref{eala-cons}} such that $\big(E,H, \inpr \big) \simeq
\EA(L, \inpr_L, D,\ta)$ for some $\Lambda$-graded invariant nondegenerate
bilinear form $\inpr_L$ on $L.$
\end{theorem}

\section{Lifting automorphisms from $\lsl_\ell(\scA)$ to $\IE(\lsl_\ell(\scA), D, C)$.}\label{sec:lift}

\subsection{The Lie algebras $\gl_\ell(\scA)$ and $\lsl_\ell(\scA)$.}\label{gl}
We assume throughout that  $\ell\ge 2$. The letter $\scA$ will always denote a
unital associative $k$--algebra. It becomes a Lie algebra $\Lie(\scA)$ with
respect to the commutator. We denote by $[\scA, \scA]$ the commutator
subalgebra of $\Lie(\scA)$,
\[
 [ \scA, \scA] = \Span_\ZZ\{ ab-ba: a,b\in \scA\}
\]
and by $\scZ(\scA) = \{ z\in \scA: [z,\scA] = 0 \}$ the centre of $\scA$,
which is also the centre of $\Lie(\scA)$. \sm

We denote by $\Mat_\ell(\scA)$ the unital associative algebra of $\ell\times
\ell$ matrices with coefficients in $\scA$, and by $\gl_\ell(\scA)$ its
associated Lie algebra: $\gl_\ell(\scA) = \Lie\big(\Mat_\ell(\scA)\big)$. \sm

The derived algebra of $\gl_\ell(\scA)$ is the {\em special linear Lie algebra
$\lsl_\ell(\scA)$ with coefficients from $\scA$}:
\begin{equation}\label{gl00}
   \lsl_\ell(\scA) = [\gl_\ell(\scA), \, \gl_\ell(\scA)].
\end{equation}
We let $\Tr$ be the trace of a matrix in $\Mat_\ell(\scA)$. The reader should
be warned that $\Tr(xy) \ne \Tr(yx)$ in general, rather we have the well-known
fact
\begin{equation} \label{gl0}
   \lsl_\ell(\scA) = \{ x\in \gl_\ell(\scA) : \Tr(x) \in [\scA, \scA]\}.\footnote{ Of course if $\scA$ is commutative, then $[\scA, \scA] = 0$ and we recover the ``usual" definition  of $\lsl_\ell(\scA).$ }
\end{equation}

Moreover, we will need
\begin{equation}
  \label{gl1}
 \scC_{\gl_\ell(\scA)} \, \big( \lsl_\ell(\scA)\big) = \scZ(\scA) \rmE_\ell =
 \scZ\big( \gl_\ell(\scA)\big),
\end{equation}
where $\scC$ denotes the centralizer and $\rmE_\ell$ the $\ell\times \ell$
identity matrix. \sm

Any $d\in \Der_k(\scA)$ stabilizes $\scZ(\scA)$ and $[\scA, \scA]$, and induces
a derivation of the associative algebra $\Mat_\ell(\scA)$ by
\begin{equation} \label{gl2}
   d \cdot x = \big( d(x_{ij}) \big) \quad \hbox{for $x=(x_{ij}) \in \Mat_\ell(\scA$).}
\end{equation}
It is then also a derivation of $\gl_\ell(\scA)$, stabilizing
$\scZ\big(\gl_\ell(\scA)\big)= \scZ(\scA) E_\ell$ and $\lsl_\ell(\scA)$. In the
following, a subalgebra $D < \Der_k(\scA)$ will be a standard feature of our
work. We will always use the action of $\Der_k(\scA)$ and hence of $D$
described in \eqref{gl2} without further explanation. Also, we will sometimes
write $dx$ or $d(x)$ for $d\cdot x$.

\subsection{The groups $\GL_\ell(\scA)$ and $\EL_\ell(\scA)$.} \label{elg}
We denote by $\GL_\ell(\scA)$ the group of invertible $\ell\times \ell$
matrices with coefficients from the unital associative $k$-algebra $\scA$.
Every $g\in \GL_\ell(\scA)$ gives rise to an automorphism $\Int(g)$ of the
associative algebra $\Mat_\ell(\scA)$, defined by $\Int(g)(a) = g a g\mn$. A
fortiori, $\Int(g)$ is an automorphism of $\gl_\ell(\scA)$. It stabilizes
$\lsl_\ell(\scA)$, whence is by restriction an automorphism of
$\lsl_\ell(\scA)$, again  denoted $\Int(g)$. Moreover, $\Int(g)$ induces the
identity on $\scZ\big( \gl_\ell(\scA) \big)$ as can be seen for the last
equality of (\ref{gl1}). \sm

The {\em elementary linear group\/} $\EL_\ell(\scA)$ is the subgroup of
$\GL_\ell(\scA)$ generated by the matrices $\rmE_\ell + a E_{ij}$ for arbitrary
$a\in \scA$ and $i\ne j$. Since $(aE_{ij})^2= 0 $ in $\Mat_\ell(\scA)$ the
derivation $\ad aE_{ij}\in \Der\big(\lsl_\ell(\scA)\big)$ is nilpotent, in fact
$(\ad aE_{ij})^3 = 0$, and the inner automorphism $\Int( E_\ell + a E_{ij})\in
\Aut_k \big(\lsl_\ell(\scA)\big)$ is elementary in the sense of \ref{elemlift}:
\[
   \Int( \rmE_\ell + a E_{ij}) = \exp (\ad a E_{ij}).\]
It follows that
\begin{equation}\label{elg1}
  \Int(g) \in \EAut_k \big( \lsl_\ell(\scA)\big)
\quad \hbox{\em for every $g\in \EL_\ell(\scA)$},
\end{equation}
where $\EAut\big( \lsl_\ell(\scA)\big)$ is the group of elementary
automorphisms of $\lsl_\ell(\scA)$. Moreover, the commutator
relation
\[ \dbl \rmE_\ell + a E_{ij}, \rmE_\ell + E_{j\ell}\dbr = E_\ell + a E_{i\ell} \quad (i,j,\ell \ne)
\]
shows that
\begin{equation}
  \label{elg2} \EL_\ell(\scA) \subset \scD \big( \GL_\ell(\scA)\big).
\end{equation}

\begin{Lemma}\label{Hlem} Let $\scA$ be a unital associative $k$--algebra
satisfying
\begin{equation}\label{Hlem00}
   \scA = \scZ(\scA) \oplus [\scA, \scA].
\end{equation}
Then
\begin{align} \label{Hlem0}
\gl_\ell(\scA) &= \scZ(A)\rmE_\ell \oplus \lsl_\ell(\scA), \quad
\hbox{and} \\
(dg) g^{-1} \in \lsl_\ell(\scA) &\iff g^{-1} dg \in \lsl_\ell(\scA) \label{Hlem1}
\end{align}
for any $d\in \Der_k(\scA)$ and $g\in \GL_\ell(\scA)$. Moreover, for $D\subset
\Der_k(\scA)$ the set
\[  H = H_D = \{ g\in \GL_\ell(\scA): (dg) g^{-1} \in \lsl_\ell(\scA) \hbox{ for all $d\in D$}\}
\]
is a normal subgroup of $\GL_\ell(\scA)$ containing the commutator subgroup
$\scD\big( \GL_\ell(\scA)\big)$ of $\GL_\ell(\scA)$.
\end{Lemma}

\begin{proof} Our assumption \eqref{Hlem00} implies $\scA \rmE_\ell = \scZ(\scA)\rmE_\ell
\oplus [\scA, \scA]\rmE_\ell$. Since $[\scA, \scA]\rmE_\ell = \scA \rmE_\ell
\cap \lsl_\ell(\scA)$ by \eqref{gl0}, the equation \eqref{Hlem0} follows from
the decomposition
\[
  x=\textstyle (\frac{1}{\ell} \Tr(x) \rmE_\ell) + (x- \frac{1}{\ell} \Tr(x) \rmE_\ell)
\]
with $\Tr(x- \frac{1}{\ell} \Tr(x) \rmE_\ell)= 0$ for arbitrary $x\in
\gl_\ell(\scA)$. \sm

The equivalence \eqref{Hlem1} is a consequence of $g^{-1} dg = (dg) g^{-1} +
[g^{-1}, dg]$ and $[g^{-1}, dg] \in \lsl_\ell(\scA)$. This shows that $g\in H
\implies g^{-1} \in H$ since
\begin{equation} \label{Hlem2}
  0 = d(g^{-1} g) = (d g^{-1} )g + g^{-1} (dg)
\end{equation}
Given that $\rmE_\ell \in H$, for $H$ to be a subgroup it suffices to show that
$g_1, g_2 \in H \implies g_1g_2 \in H$. But this follows from
\[
   \big( d(g_1g_2)\big) (g_1 g_2)\mn =
(dg_1) g_2 g_2\mn g_1\mn + g_1 (dg_2) g_2\mn g_1\mn = (dg_1)g_1\mn +
\Int(g_1) \big( (dg_2)g_2\mn\big)
\]
since $\Int(g_1)$ stabilizes $\lsl_\ell(\scA)$.  Thus $H$ is a subgroup, and it
will be a normal subgroup as soon as we have shown that $H$ contains any
commutator $\dbl g_1, g_2\dbr$ where $g_1, g_2 \in \GL_\ell(\scA)$. We have
\begin{align*}
&\big(d \dbl g_1, g_2\dbr \big) \dbl g_1, g_2\dbr\mn  =
  d\big( g_1 g_2 g_1\mn g_2\mn \big) \, g_2 g_1 g_2\mn g_1\mn \\
&= (dg_1)\, (g_2 g_1\mn g_2\mn g_2 g_1 g_2\mn g_1\mn)
      + g_1 \big( (dg_2) g_1\mn g_2\mn g_2 g_1 g_2\mn \big) g_1\mn \\
  &\qquad  + g_1g_2 \big( d(g\mn)g_2\mn g_2 g_1\big)g_2\mn g_1\mn
      + g_1 g_2 g_1\mn \big( d(g_2\mn g_2) \big)g_1 g_2\mn g_1\mn  \\
 &= (dg_1) g_1\mn + \Int(g_1) ( (dg_2) g_2\mn )
 + \Int(g_1g_2) \big((dg_1\mn) g_1\big) + \Int(g_1g_2 g_1\mn) \big( (dg_2\mn) g_2\big)
\end{align*}
To proceed, we use \eqref{Hlem0}, thus uniquely writing any $x\in
\gl_\ell(\scA)$ as $x=x_z + x_s$ with $x_z\in \scZ\big(\gl_\ell(\scA)\big)$ and
$x_s\in \lsl_\ell(\scA)$. Decomposing $y\in \gl_\ell(\scA)$ in the same way, we
have
\begin{equation} \label{Hlem3}
 (xy)_z = (yx)_z
\end{equation}
since $xy = yx + [x,y]$ with $[x,y]\in \lsl_\ell(\scA)$. Because $\Int(g)$
stabilizes $\lsl_\ell(\scA)$ and satisfies $\Int(g) (zE_\ell) = zE_\ell$ for
$z\in \scZ(\scA)$ we now get
\begin{align*}
  &\big( (dg_1) g_1\mn\big)_z + \Big(\Int(g_1g_2)\big((dg_1\mn) g_1\big)\Big)_z
  =\big( (dg_1) g_1\mn\big)_z + \Int(g_1g_2)\big((dg_1\mn) g_1\big)_z
\\  & = \big( (dg_1) g_1\mn\big)_z + \big((dg_1\mn) g_1\big)_z
   =\big( (dg_1) g_1\mn\big)_z +\big(g_1 (dg_1\mn) \big)_z
    = \big( d(g_1 g_1\mn)\big)_z =  0,
\end{align*}
thus proving that $(dg_1) g_1\mn + \big(\Int(g_1g_2)\big((dg_1\mn) g_1\big) \in
\lsl_\ell(\scA)$. Similarly
\[
     \Int(g_1) ( (dg_2) g_2\mn ) + \Int(g_1g_2 g_1\mn) \big( (dg_2\mn) g_2\big) \in \lsl_\ell(\scA).
\]
 Hence $\dbl g_1, g_2 \dbr \in H$, and therefore
$\scD\big(\GL_\ell(\scA)\big) \subset H$. \end{proof}

\subsection{Interlaced extensions based on $\lsl_\ell(\scA)$.}
\label{iia} We specialize the setting of \ref{gen-data} to $L=\lsl_\ell(\scA)$
with the aim of constructing a suitable interlaced extension that will allow us
to lift the automorphisms used in conjugacy. Being an interlaced extension, we
need to specify data $(\beta, D, C, \tau)$. \ms

\begin{inparaenum}[(i)]
\item We fix a linear form
\begin{equation} \label{iia1} \eps \co \scA \to k, \quad \eps([\scA,
\scA]) = 0
\end{equation}
and define $\be=\be_\eps \co L \times L \to k$ by
 \begin{equation} \label{iia2}
   \beta_\eps (x,y) = \eps \big(\Tr(xy)\big)
     = \textstyle \sum_{i,j=1}^\ell \eps (x_{ij} y_{ji})
\end{equation}
for $x=(x_{ij}$ and $y=(y_{ij})$. Then $\be$ is an invariant bilinear form on
$L$, and every invariant bilinear form on $L$ is of the type $\be_\eps$ for a
unique linear form $\eps$ satisfying \eqref{iia1}
(\cite[7.10]{n:persp}). \sm

\item We let $D$ be a subalgebra of derivations of $\scA$, which are
    skew with respect to  the bilinear form $(a,b) \mapsto \eps(ab)$,

       \[ D < \SDer_k(\scA),  \]
and let $D$ act on $L$ as in \eqref{gl2}. Then $D$ acts on $L$ by skew
derivations with respect to $\be$.

\item We choose $C\subset D^*$ and $\ta$ as in \eqref{gen-dataiii} and
    \eqref{gen-dataiv} of \ref{gen-data}. \end{inparaenum} \ms

\noindent Using these data we form the interlaced extension
\[
     \IE(L, \be_\eps,  D, C, \tau) =  E = L \oplus C \oplus D.
\]

\subsection{Enlarging interlaced extensions.}\label{vergross}  To suitably enlarge an interlaced
extension $E=\IE(L, D, C)$ with $L=\lsl_\ell(\scA)$ as in \ref{iia}, we embed
$L$ into $L'=\lsl_{\ell+m}(\scA)$, $m\in \NN$ arbitrary, via
\begin{equation} \label{etap1}
    \lsl_\ell(\scA) \to  \lsl_{\ell+m}(\scA), \quad l \mapsto \begin{pmatrix}
  l & 0 \\ 0 & 0 \end{pmatrix}.
\end{equation}
Following the outline of \ref{enl} we next need an invariant bilinear form
$\be'$ on $L'$. We take $\be' = \be'_\eps$ as defined in \eqref{iia2}:
$\be'(x', y') = \eps\big(\Tr(x'y')\big) = \sum_{i,j=1} ^{\ell+m}
\eps(x'_{ij}y'_{ji})$ for $x'=(x'_{ij})$ and $y'=(y'_{ij}) \in
\lsl_{\ell+m}(\scA)$. Then the condition \eqref{enl-ii} of \ref{enl} is
fulfilled: $\be'(l_1, l_2) = \be(l_1, l_2)$ for $l_1, l_2 \in L$. \sm

We also have condition \eqref{enl-iii} of \ref{enl}, i.e., $D$ acts on $L'$
by skew derivations extending the action of $D$ on $L$. Finally,
\ref{enl}\eqref{enl-iv} also holds. Indeed, for $x'$, $y' \in L'$ as before
and $d\in D$, we have
\begin{align*}
  &\si_{D, \be'} (x', y')(d) = \be'( d \cdot x', y')
   = \textstyle  \sum_{i,j=1}^{\ell+m} \; \eps\big( (d\cdot x'_{ij}) y'_{ji})
 \\ & \quad = \textstyle \sum_{i,j=1}^{\ell+m} \be( d\cdot (x'_{ij}E_{12}), y'_{21} E_{21})
  =  \Big(\sum_{i,j=1}^{\ell+m} \si_{D, \be} ( x'_{ij}E_{12}, y'_{ji} E_{12})\Big)(d)
\end{align*}
which shows $\si_{D, \be'}(x', y') \in C$. In sum, we have shown that {\em for
any $m\in \NN$ the interlaced extension $E= \IE(L, \be_\eps, D, C, \tau)$ is a
subalgebra of $E' = \IE(\lsl_{\ell+m}(\scA), \be'_\eps, D, C, \tau)$.} \sm

We are now ready to prove the main result of this section.

\begin{theorem}\label{etap} Let $\scA$ be a unital associative $k$-algebra satisfying $\scA = \scZ(\scA) \oplus [\scA, \scA]$, and let $E=\IE(L,D,
C)$ be an interlaced extension based on $L=\lsl_\ell(\scA)$ as specified in
{\rm \ref{iia}}. Assume that $g\in \GL_\ell(\scA)$ is\/ {\em stably
elementary\/} in the sense that there exists $m\in \NN$ such that
\[
   g' = \begin{pmatrix}
     g & 0 \\ 0 & \rmE_m   \end{pmatrix} \in \EL_{\ell+m}(\scA).
\]
Then the automorphism $\Int(g)$ of $\lsl_\ell(\scA)$ lifts to an automorphism
of $E$.
\end{theorem}

\begin{proof} We embed $L=\lsl_\ell(\scA)$ into $L'= \lsl_{\ell+m}(\scA)$ as in
\eqref{etap1}. We then know that $E$ can be enlarged to an interlaced extension
$E'=\IE(L', D, C)$. Moreover, by \eqref{elg1} and Proposition~\ref{elemlift}
the elementary automorphisms $\Int(g')$ of $L'$ lifts to a special automorphism
$f'$ of $E'$, determined by maps $\Int(g') = f_{L'} \in \Aut_k (L')$ and linear
maps $\eta'\co D \to L'$, $\psi'\co L' \to C$ and $\vphi'\co D \to C$ as in
Lemma~\ref{autchar}. It
will be sufficient to show $f'(E) = E$. Since $f'_{L'}(L) = \Int(g)(L) = L$, it
is in view of \eqref{enl2} enough to prove
\[
   \eta'(d) \in L \qquad \hbox{for all $d\in D$}.
\]

By \ref{autchar}\eqref{autochar3} we have
\begin{equation} \label{etap2}
 d \cdot f_{L'}(l') - f_{L'} (d \cdot l') = [ f_{L'}(l'), \, \eta'(d)]
\end{equation}
for all $d\in D$ and $l'\in L'$. For $l \in L$ we know $f_{L'}(l) =
\Int(g)(l) \in L$ and also $d \cdot f_{L'}(l) - f_{L'} (d \cdot l) \in L$. It
thus follows from \eqref{etap2} for $l'=l\in L$ that  $\eta'(d)$ normalizes
$L$. One easily calculates that then $\eta'(d)$ has the form
\[ \eta'(d) =\begin{pmatrix} \al & 0 \\ 0 & \beta
\end{pmatrix} , \quad \al \in \gl_\ell(\scA), \be \in \gl_m(\scA)\]
(we have suppressed in our notation that $\al$ and $\be$ depend linearly on
$d$). Employing the obvious subdivision for matrices $l'\in L'$,
\[
   l' = \begin{pmatrix} x_1 & x_2 \\ x_3 & x_4 \end{pmatrix}, \quad
    x_1 \in \gl_\ell(\scA),
\]
we get
\[
   f_{L'}(l') = \Int(g')(l') =
      \begin{pmatrix} g x_1 g\mn & g x_2 \\ x_3g \mn & x_4 \end{pmatrix},
\]
whence for $d\in D$ the left hand side of \eqref{etap2} becomes
\[
d \cdot f_{L'}(l') - f_{L'} (d \cdot l') =
\begin{pmatrix} (dg) x_1 g\mn + gx_1 d(g\mn)& (dg) x_2 \\ x_3d(g \mn) & 0\end{pmatrix},
\]
while the right hand side of \eqref{etap2} is
\[
 [ f_{L'}(l'), \, \eta'(d)]
 = \begin{pmatrix} [g x_1 g\mn, \al] & g x_2\be - \al gx_2 \\
             x_3g \mn \al - \be x_3g\mn & [x_4, \be] \end{pmatrix}.
\]
Thus
\begin{align}
(dg) x_2 &=g x_2\be - \al gx_2  \label{etap4} \\
 0&= [x_4, \be] \label{etap5}
\end{align}
Since every $x_4\in \gl_m(\scA)$ is part of some matrix $l'\in L'$, it follows
that \eqref{etap5} holds for all $x_4\in \gl_m(\scA)$. Therefore, by
\eqref{gl1}, \[ \be = z \rmE_m\] for some $x\in \scZ(\scA)$. We substitute this
expression for $\be$ into \eqref{etap4} and obtain $(dg) x_2 = (zg-\al g) x_2$.
Since this holds for all $x_2 \in \Mat_{mn}(\scA)$ we get $dg = zg - \al g$ or
\[ \al = z \rmE_\ell - (dg) g\mn.
\]
Because $g'\in \EL_{\ell+m}(\scA)$ it follows from \eqref{elg2} and
Lemma~\ref{Hlem} that $(dg')(g')\mn \in \lsl_{\ell+m}(\scA)$ for all $d\in D$.
But
\[
    (dg') (g')\mn = \begin{pmatrix} (dg) g\mn & 0 \\ 0 & 0 \end{pmatrix},
\]
so that $(dg) g\mn \in \lsl_\ell(\scA)$ follows. Since $\eta'(d) \in
\lsl_{\ell+m}(\scA)$ we now get \begin{align*}
  \Tr\big( \eta'(d)\big) &= \Tr(\al) + \Tr(\be) =
   \ell z -  \Tr\big( (dg)g\mn\big) + m z \\ &=
   (\ell+m) z -\Tr\big( (dg)g\mn\big) \in[\scA, \scA].
\end{align*}
As $\scA = \scZ(\scA) \oplus [\scA, \scA]$ by assumption and $\Tr\big(
(dg)g\mn\big) \in[\scA, \scA]$, this forces $(\ell+m) z = 0$ so that $z=0$ and
finally $\be = 0$, i.e., $\eta'(d) \in L$ follows.
\end{proof}

\subsection{Quantum tori (review)}\label{qua-pro}
We will later specialize $\scA=\scQ$ to be a quantum torus. Why we do so, is
explained in \ref{link}: $\lsl_\ell(\scQ)$ is then a centreless Lie torus. In
this subsection we review some properties of quantum tori that we will use.
Contrary to the standing assumption for this paper, in this subsection our base
field $k$ can have arbitrary characteristic. We let $\La$ be a free abelian
group of rank $n$. \ms

\begin{inparaenum}[(a)]
\item \label{qua-pro1}(Definition) By definition, a {\em quantum torus (with
    grading  group $\La$)\/} is an associative  unital $\La$-graded
    $k$-algebra   $\scQ= \bigoplus_{\la \in \La} \scQ^\la$ such that

\qquad (QT1) $\dim \scQ^\la \le 1$ for all $\la \in \La$,

\qquad (QT2) every $0 \ne a \in \scQ^\la$ is invertible, and

\qquad (QT3) $\La$ is generated as abelian group by $\{ \la \in \La: \scQ^\la
\ne 0\}$. \sm

Since the invertible elements of an associative algebra form a group, $\{\la
\in \La: \scQ^\la \ne 0 \}$ is a subgroup of $\La$, whence equals $\La$ by
(QT3).  \sm

\item After fixing a basis $\beps= (\eps_i)$ of $\La$, we can choose $0 \ne x_i
    \in \scQ^{\eps_i}$ and then get a quantum matrix $q=(q_{ij}) \in \Mat_n(k)$
    defined by $x_i x_j = q_{ij} x_j x_i$. We recall that $q=(q_{ij}) \in
    \Mat_n(k)$ is called a {\em quantum matrix\/} if $q_{ij} = q_{ji}^{-1}$ and
    $q_{ii} = 1$ for all $1\le i,j\le n$.

    Then, using $x_i^{-1} = $ the
    inverse of $x_i$, we define $x^\la = x_1^{\ell_1} \cdots x_n^{\ell_n}$ for
    $\la = \ell_1 \eps_1 + \cdots + \ell_n \eps_n \in \La$:
\begin{equation}
  \label{qua-pro11} \scQ = \textstyle \bigoplus_{\la \in \La} k x^\la.
\end{equation}
One can then also realize a quantum torus as the unital associative $k$-algebra
presented by generators $x_1, \ldots, x_n , x_1^{-1}, \ldots, x_n^{-1}$ and
relations \[ x_i x_i^{-1} = 1_\scQ = x_i^{-1} x_i, \quad x_i x_j = q_{ij} x_j
x_i.
\]
We will refer to this view of $\scQ$ as a {\em coordinatization}. \sm

\item\label{qua-pro2} The centre of $\scQ$ is a $\La$-graded
    subalgebra,
\[
 \scZ(\scQ) = \textstyle \bigoplus_{\xi  \in \Xi} \scQ^\xi
\]
where $\Xi$ is the so-called {\em central grading group\/}:
\[\Xi = \{ \la \in
\La : \scQ^\la \subset \scZ(\scQ)\}.\] This is a free abelian group of rank
$z\le n$. Hence $\scZ(\scQ)$ is a Laurent polynomial ring in $z$ variables,
which we may take as $t_1, \ldots, t_z$ (these can be taken to be of the form
$x^\la$ for suitable $\la$'s). \sm

\item \label{qua-pro-e} The grading properties of a quantum torus $\scQ$ show
    that     $\scQ$    is    {\em fgc\/} in the sense that $\scQ$ is finitely generated
    as a module over $\scZ(\scQ)$ if and only if $\Xi$  has  finite index in $\La$.
    Equivalently, for some (hence all) coordinatization all entries $q_{ij}$ of the quantum
    matrix $q$ have finite order. If this holds, then for every coordinatization the
    $q_{ij}$ have finite order. \sm

\item\label{qua-pro3}  We define
\[[\scQ,\scQ] = \Span_k \{[a,b]: a,b\in \scQ\},\] a
    graded subspace of $\scQ$. One knows (see e.g.\ \cite[Prop.~2.44(iii)]{bgk}
    for $k=\mathbb{C}$ or \cite[(3.3.2)]{ny} in general)
\begin{equation}\label{qua-pro33}
   \scQ = \scZ(\scQ) \oplus [\scQ,\scQ]. \end{equation}

\item\label{qua-pro-unit} An element $u$ of $\scQ$ is invertible if and only if
    $0\ne u \in \scQ^\la$ for some $\la \in \La$. \sm

\item\label{qua-pro-der} The derivation Lie algebra $\Der_k(\scQ)$ is graded:
    $\Der_k(\scQ) = \bigoplus_{\la \in \La}\, \Der_k(\scQ)^\la$ where
    $\Der_k(\scQ)^\la$ consists of those derivations $d$ satisfying
    $d(\scQ^\mu) \subset \scQ^{\la +\mu}$ for all $\mu \in \La$. The inner derivations
    of $\scQ$ are the maps $\ad q $, given by $\ad(q)(q') =
    qq'-q'q$ for $q,q'\in \scQ$. They form a graded ideal $\IDer \scQ = \{ \ad q: q \in \scQ\}$ of
    $\Der_k(\scQ)$.     As in \ref{lietor-prop}\eqref{lietor-prop-e}, the grading $\scQ =
    \bigoplus_{\la \in \La}\, \scQ^\la$ gives rise to {\em degree
    derivations\/} $\pa_\theta$ of $\scQ$, defined by $\pa_\theta(q) =
    \theta(\la) q$ for $\theta \in \Hom_\ZZ(\La, k)$ and $q \in
    \scQ^\la$. We put $\euD_\scQ= \{\pa_\theta : \theta \in \Hom_\ZZ(\La,k)\}$
    and define
    \[ \CDer(\scQ) = \scZ(\scQ)\, \euD_\scQ = \ts \bigoplus_{\xi \in \Xi}\,
     \scQ^\xi \, \euD_\scQ,
     \] the graded subalgebra of {\em centroidal derivations}. Then
     (\cite[Cor.~2.3]{OP})
          \[ \Der_k(\scQ) = \IDer(\scQ) \rtimes \CDer(\scQ), \quad
          \hbox{so } \ts
          \IDer(\scQ) = \bigoplus_{\la \not\in \Xi}\, \Der_k(\scQ)^\la. \]

Let $\eps \co \scQ \to k$ be the linear form defined by $\eps_(1_\scQ) =
    1$
 and $\eps(\scQ^\la) = 0$ for $\la \ne 0$. The skew-symmetric derivations with
 respect to the bilinear form $(q,q') \mapsto \eps(qq')$ have the following
 description:
\begin{align}
  \SDer(\scQ) &= \SCDer(\scQ) \oplus \IDer(\scQ) \nonumber \\
   \SCDer(\scQ) &= \SDer(\scQ) \cap \CDer(\scQ) = \ts
       \bigoplus_{\xi \in \Xi}\, \SCDer(\scQ)^\xi \quad\hbox{where}\nonumber\\
  \SCDer(\scQ)^\xi &= \scQ^\xi \, \{\pa_\theta : \theta \in \Hom_\ZZ(\La, k),
                            \theta(\xi) = 0 \}. \label{cdrq}
\end{align}
\end{inparaenum}

\subsection{$\lsl_\ell(\scA)$ as Lie torus.} \label{link} In this subsection we
describe for which algebras $\scA$ the Lie algebra $\lsl_\ell(\scA)$ is a Lie
torus as defined in \ref{def:lietor} and identify the data \ref{eala-cons}
necessary to construct an EALA with centreless core $\lsl_\ell(\scA)$. All
un-attributed result can be found in \cite[\S7]{n:persp} or are easily verified
by the reader. We assume $\scA \ne 0$ throughout. \sm

\begin{inparaenum}[(a)]

\item Let $\De$ be the root system of type ${\rm A}_{\ell-1}$, realized as $\De
    = \{ \eps_i - \eps_j, 1\le i,j\le \ell\}$ in standard notation. Then the
    Lie algebra $\lsl_\ell(A)$ has a canonical grading by the root lattice
    $\euQ(\De)$,
\begin{align} \label{link1}
 \lsl_\ell(\scA) &= \ts \bigoplus_{\al \in \De} \, \lsl_\ell(\scA)_{\alp}, \qquad \hbox{for}
     \\
    \lsl_\ell(\scA)_\al &= \begin{cases} \scA\, E_{ij}, & \alp = \eps_i - \eps_j \ne 0, \\
                          \{ x\in \lsl_\ell(\scA): x \hbox{ diagonal}\}, & \alp = 0.
                          \end{cases} \nonumber
\end{align}

\item\label{link-b} Let $e=a E_{ij} \in \lsl_\ell(\scA)$ for $i\ne j$. Then $e$
    is part of an $\lsl_2$-triple $(e, h, f)$ satisfying
$[h,x_\be] = \lan \beta, \alp\ch\ran x_\be$ for all $\beta \in \De$ and $x_\be
\in L_\be$ if and only if $a$ is invertible in $\scA$. In this case $f=a^{-1}
E_{ji}$
    and $h=E_{ii} - E_{jj}$. \sm

\item \label{link-d} Let $\La$ be an abelian group, and let $\scA =
    \bigoplus_{\la \in \La} \scA^\la$ be a $\La$-graded unital associative
    $k$-algebra. Then the $\euQ(\De)$ grading \eqref{link1} of
    $\lsl_\ell(\scA)$ extends to a $(\euQ(\De)\oplus \La)$-grading of $\lsl_\ell(\scA)$,
\[ \lsl_\ell(\scA) = \ts \bigoplus_{\alp \in \De, \, \la \in \La} \,
\lsl_\ell(\scA)^\la_\al
\]
by letting $\lsl_\ell(\scA)^\la_\al$ consist of those matrices, for which all
entries lie in $\scA^\la$. Conversely, a $\euQ(\De) \oplus \La$-grading of
$\lsl_\ell(\scA)$ extending the $\euQ(\De)$-grading \eqref{link1} arises from a
$\La$-grading of the associative algebra $\scA$ as described above. \sm

\item Because of \eqref{link-d}, for $\lsl_\ell(\scA)$ to satisfy the axiom
    (LT1)     of    \ref{def:lietor} with $\euQ(\De)$-grading \eqref{link1} it is necessary and
    sufficient for the associative $k$-algebra $\scA$ to be $\La$-graded.
    Observe that then also (LT2.b) holds since $0 \ne 1_\scA \in \scA^0$ and
    therefore $0 \ne 1_\scA E_{ij} \in \lsl_\ell(\scA)_\al^0$ for $\al= \eps_i
    - \eps_j \ne 0$. Because of \eqref{link-b}, the axiom (LT2.a) holds if and
    only if $\scA= \scQ = \bigoplus_{\la \in \La}\,  \scQ^\la$ is a quantum
    torus.

    Since (LT3) is clear, (LT4) says that $\lsl_\ell(\scQ)$ is a Lie torus of type $(\De,\La)$
    if and only if $\scQ$ is a quantum torus of type $\La$, as defined in
    \ref{qua-pro}\eqref{qua-pro1}. In this case, it follows from \eqref{link-cent} that
    $L$ is fgc as defined in \ref{lietor-prop} if and only if $\scQ$ is an fgc
    quantum torus in the sense of \ref{qua-pro}\eqref{qua-pro-e} -- but we will not
    assume this in the following. \sm

\item In the remainder of this subsection we let $L=\lsl_\ell(\scQ)$ for
    $\scQ$    a    quantum    torus    with    grading group $\La$. Because    of    \eqref{qua-pro33}, the assumption of Lemma~\ref{Hlem} is
    fulfilled. Then \eqref{Hlem0} and \eqref{gl1} imply that $L$ is a
    centreless Lie torus of type $(\De,\La)$. \sm

    Hence, by Theorem~\ref{n:mainconst}, $L$ is centreless core of an EALA obtained
    by the construction~\ref{eala-cons}. We describe the bilinear forms
    $\inpr$ and derivation algebras $D$ allowed in this construction in the next two
    items. \sm

\item Every $\La$-graded invariant symmetric bilinear form $\be$ on $L$ has the
    form \eqref{iia2} where $\eps\co \scQ \to k$ is a linear form vanishing on
    $ \bigoplus_{0 \ne \la }\, \scQ^\la$ and is therefore given by the scalar
    $\eps(1_\scQ)$ which we can assume to be $1\in k$. \sm

\item \label{link-cent}\footnote{ The items \eqref{link-cent} and
    \eqref{link-der}
    are     true for any algebra $\scA$ in place of $\scQ$}
    For    $z\in
    \scZ(\scQ)$     define    $\chi_z    \in    \End_k    \big(
    \Mat_\ell(\scQ)) $ by  $\chi_z \, (x) = (z x_{ij})$ for $x= (x_{ij}) \in \Mat_\ell(\scQ)$.
    Then $\chi_z$ stabilizes $\lsl_\ell(\scQ)$ and defines by restriction a
    centroidal transformation of $\lsl_\ell(\scQ)$. The map
 $\scZ(\scQ) \to \Ctd\big( \lsl_\ell(\scQ) \big)$, $z \mapsto \chi_z$,
 is an isomorphism of $k$-algebras. \sm

\item \label{link-der} For $d\in \Der(\scQ)$ we denote by $\Mat_\ell(d)$ the
    derivation     of
    $\lsl_\ell(\scQ)$ defined in \eqref{gl2}. The maps $d \mapsto \Mat_\ell(d)$ is clearly
    a monomorphism of Lie algebras. Moreover,
    \begin{align*}
      \Der_k\big(\lsl_\ell(\scQ)\big) &= \IDer\big(\lsl_\ell(\scQ)\big)
            + \Mat_\ell\big(\Der(\scQ)\big) \\
      \Mat_\ell\big( \IDer(\scQ)) &= \IDer\big( \lsl_\ell(\scQ)\big) \cap
                \Mat_\ell\big( \Der(\scQ)\big) \simeq \IDer(\scQ) \\
       \CDer\big(\lsl_\ell(\scQ) \big) & = \Mat_\ell\big( \CDer(\scQ)\big)
           \simeq \CDer(\scQ) \\
      \SDer\big(\lsl_\ell(\scQ)\big) &= \IDer\big( \lsl_\ell(\scQ) \big)
             + \Mat_\ell\big( \SDer(\scQ)\big) \\
    \SCDer\big(\lsl_\ell(\scQ)\big)&= \Mat_\ell(\SCDer(\scQ)\big) \simeq \SCDer(\scQ)
    \end{align*}
for $\Der(\scQ)$, $\IDer(\scQ)$, $\CDer(\scQ)$, $\SDer(\scQ)$ and
$\SCDer(\scQ)$ described in \ref{qua-pro}\eqref{qua-pro-der}. Note that the
first three equations above together with \ref{qua-pro}\eqref{qua-pro-der}
prove \eqref{lietor-prop1} for the case $L=\lsl_\ell(\scQ)$. \sm

\item The maximal possible choice for $D$ in the construction \ref{eala-cons}
    is $\SCDer \big(\lsl_\ell(\scQ)\big)$ which we identify with $\SCDer(\scQ)$
    using the isomorphism $\Mat_\ell$ of \eqref{link-der}. For $\scQ=k[x_1^{\pm}, \ldots, x_n^{\pm
    1}]$, $n\ge 2$, a non-zero affine cocycle $\ta$ has been exhibited in \cite[Rem.~3.71]{bgk}.
    It can be described as follows.

    Modulo the isomorphism $\Mat_\ell$ of \eqref{link-der} we identify $\SCDer \big(\lsl_\ell(\scQ)\big)$
    with $\SCDer(\scQ)$. Denoting by $\lan \cdot , \cdot \ran$ the standard inner
    product of $k^n$ and using the natural embedding $\ZZ^n \subset k^n$ we
    can further identify
    \begin{align*}
    \SCDer (\scQ) &= \ts \bigoplus_{\la \in \La = \ZZ^n}\, \SCDer(\scQ)^\la,
    \quad \hbox{where for $\la = (\la_1, \ldots, \la_n) \in \ZZ^n$} \\
      \SCDer(\scQ)^\la &\equiv  \{u =(u_i) \in k^n : \ts \sum_{i=1}^n u_i \la_ i = 0 \} =:D^\la,
       \end{align*}
    cf.\ \eqref{cdrq}.  $u_\al \in D^\al$, $v_\be \in D^\be$ and $w_\ga \in D^\ga$
    define
 \[ \ta(u_\al,  v_\be)\, (w_\ga) = \begin{cases} \al(v)\, \beta(w) \, \ga(u)
                & \hbox{if $\al + \be + \ga = 0$}, \\
                0 & \hbox{otherwise,}\end{cases}
\]
Then $\ta$ is an affine cocycle. It is non-trivial in the sense that the EALAs
associated with $L=\lsl_\ell(\scQ)$, $\ell\ge 3$, $D=\SCDer(L)$ and the two
affine cocycles $\ta$ as above respectively $\ta=0$ are not isomorphic
(\cite[Thm.~5.76]{Kr}.
\end{inparaenum}

\section{Proof of the Main Theorem} \label{sec:promt}

The proof of our main result will be based on the computation of $K$-Theory of
noncommutative (twisted) Laurent polynomial rings due to  D.~Quillen. We
first  briefly recall functors $K_0$ and $K_1$. A nice introduction to the
subject can be found in \cite{R} and  \cite{Wb}.

\subsection{$K_0(\scA)$ and $K_1(\scA)$ for a ring $\scA$.}\label{kkrev}
Let $\scA$ be a ring (unital, but not necessarily commutative). If $P$ is a
(left) $\scA$-module, we denote its isomorphism class by $[P].$ Consider the
free abelian group $FK_0(\scA)$ generated by the set of isomorphism classes of
projective $\scA$-modules of finite type. Then $K_0(\scA)$ is the quotient of
the group $FK_0(\scA)$ by the normal subgroup generated by  the relation
$$ [P] = [P'] + [P'']$$
whenever there exists an exact sequence of $\scA$-modules
$$0 \to P' \to P \to P'' \to 0.$$

As in \ref{elg} we denote by $\GL_\ell(\scA)$, $\ell \in \NN_+$, the group of invertible
$\ell \times \ell$ matrices with entries in $\scA$. For each $m\in \NN_+$ we
have a natural embedding $\GL_\ell(\scA)\hookrightarrow \GL_{\ell + m}(\scA)$
given by
\begin{equation}
X\longrightarrow
\begin{pmatrix}
X & 0\\
0 & \rmE_{\ell + m}
\end{pmatrix}, \label{kkrev1}
\end{equation}
cf.\ \eqref{etap1} for the corresponding embedding on the level of Lie
algebras. We let $\GL_{\infty}(\scA)$ be the direct limit of $\GL_\ell(\scA)$
with respect to  the embeddings \eqref{kkrev1}. Again as in \ref{elg},  we
let $\EL_\ell(\scA)$ be the elementary linear subgroup of $\GL_\ell(\scA)$ and
let $\EL_\infty(\scA)$ be the direct limit of the $\EL_\ell(\scA)$. Then
$$
K_1(\scA)=\GL_\infty(\scA)/[\GL_\infty(\scA),\GL_\infty(\scA)]
   = \GL_\infty(\scA)/\EL_\infty(\scA)
$$
(the first equality is the standard definition of $K_1(\scA)$, while the second
equality is a classical theorem of Whitehead.) \medskip
\begin{remark} The construction of $K_0$ and $K_1$ is functorial on $k$-algebras.
Given a $k$-algebra homomorphism $\eta : \scA \to \scB$ we will denote by
$\eta^*$ the induced group homomorphisms $K_0(\scA) \to K_0(\scB)$ and
$K_1(\scA) \to K_1(\scB).$
\end{remark}

Next we recall the definition of noncommutative Laurent polynomial ring
$\scA_{\phi}[t^{\pm 1}]$. Consider an automorphism $\phi$ of a (unital,
associative and not necessarily commutative) $k$-algebra $\scA$. The
multiplication in $\scA$ will be denoted by juxtaposition. We define a new
unital and associative $k$-algebra $\scA_\phi[t^{\pm 1}]$ as follows. The
underlying $k$-vector space structure is the free left $\scA$-module with basis
$\{t^m\}_{m \in \ZZ}$. The multiplication on $\scA_\phi[t^{\pm 1}]$, which we
will denoted by $\cdot$, is given by
\begin{equation}\label{multiplication}
\Sigma_{i \in \ZZ} \, a_it^i \cdot \Sigma_{j \in \ZZ} \, a'_jt^j
 = \Sigma_{i,j \in \ZZ} \, a_i \phi^i(a'_j)t^{i + j} \quad \text{\rm for all} \,\,\,  a_i, a'_j \in \scA.
\end{equation}
It is known that if $\scA$ is noetherian (resp.\ regular), so  is
$\scA_\phi[t^{\pm 1}]$ (see \cite{arta} Prop. 2.21). We also observe that
$\phi$ induces a natural action $\phi^\ast$ on $K_0(\scA)$ and $K_1(\scA)$.
Namely, if $P$ is a projective $\scA$-module then
$\phi^\ast([P]):=[P\otimes_\phi \scA]$. It is obvious that if $P$ is free or
projective of finite type,  so is $\phi^\ast(P)$. Also, for every matrix
$X=(x_{ij})$ in $\GL_\ell(\scA)$ we let $\phi^\ast(X)=(\phi(x_{ij}))$. This of
course induces an action $\phi^\ast$ on $\GL_\infty(\scA)$  stabilizing
$\EL_\infty(\scA)$, and hence also an action  on $K_1(\scA)$. \medskip

The following result is  due to  D. Quillen \cite[\S 6, page 122]{Q}.

\begin{theorem} \label{qui}
Let $\phi \in \Aut_k(\scA)$. Assume that $\scA$ is noetherian and regular. Let
$\eta : \scA \to \scA_{\phi}[t^{\pm 1}]$ be the canonical embedding of $k$-algebras. Then the  following
sequence of abelian groups
\begin{equation}\label{Q}
\xymatrix{
  K_1(\scA)\ar[r]^<<<<{1 -\phi^*} & K_1(\scA) \ar[r]^{\eta^*} & K_1(\scA_{\phi}[t^{\pm 1}])
         \ar[r]^{\partial}
         & K_0(\scA) \ar[r]^{1 - \phi^*} & K_0(\scA) \ar[r]^{\eta^*} & K_0(\scA_{\phi}[t^{\pm 1}]) \ar[r] & 0
}
\end{equation}
is exact.\footnote{The maps $\phi^*$ and $\eta^*$ have been defined already.
The nature of $\partial$ is explained in Quillen's paper.} \ms
\end{theorem}

We will apply Theorem~\ref{qui} to a quantum torus $\scQ$, Thus, as explained
in \ref{qua-pro}, we can view
$\scQ$ as the  unital associative $k$-algebra presented by generators $x_1,
\ldots, x_n , x_1^{-1}, \ldots, x_n^{-1}$ and relations $ x_i x_i^{-1} = 1_\scQ
= x_i^{-1} x_i$, $x_i x_j = q_{ij} x_j x_i$, where the $q_{ij}$ are non-zero
elements of $k$, $q_{ii}=1$  and $q_{ij}=q_{ji}^{-1}$. For convenience in what
follows we assume that the elements $q_{ij}$ are fixed throughout our
discussion, and write
$$\scQ = k[x_1^{\pm 1}, \cdots, x_n^{\pm 1}].$$
It is immediate from the defining relations that the $k$-vector space $\scQ$ is
a direct sum $\scQ = \bigoplus_{i_1, \ldots, i_n\in \ZZ}\, k x_1^{i_1} \cdots
x_n^{i_n}$.

The quantum torus $\scQ$ contains a subring
\[
 \scQ_{n-1}=k[x_1^{\pm 1},\ldots,x_{n-1}^{\pm 1}]
\]
generated by $x_1^{\pm 1},\ldots,x_{n-1}^{\pm 1}$. Obviously, the conjugation
by $x_n$  stabilizes $\scQ_{n-1}$ and thus induces an automorphism $\phi$
on $\scQ_{n-1}$ so that we may view $\scQ$ as a noncommutative Laurent
polynomial ring $\scQ=\scA_{\phi}[x_n^{\pm 1}]$ where $\scA=\scQ_{n-1}$. The
advantage of realizing $\scQ$ in this form is that it allows us to compute
$K_0(\scQ)$ and $K_1(\scQ)$ by induction on $n.$ We start with computing
$K_0(\scQ)$.

\begin{lemma}\label{K0}
The group $K_0(\scQ)$ is isomorphic to $\mathbb{Z}$. Its generator is the class of a free $\scQ$-module  of rank $1$.
\end{lemma}

This is \cite[Thm.~3.17]{arta}. We include a short proof for the sake of completeness.

\begin{proof}
We reason by induction on $n$. If $n =1 $ then $\scQ = k[x^{\pm 1}]$  is a
commutative Laurent polynomial ring. Since $\scQ$ is then a principal
ideal domain, every projective $\scQ$-module is  free. Our result is then clear.

Assume $n > 1.$ Consider the natural $k$-algebra inclusion $\eta: \scQ_{n-1} \to \scQ_n.$ By induction we may assume that $K_0(\scQ_{n-1})\simeq \mathbb{Z}$. Since $\phi^\ast$ acts trivially
on its generator, it acts trivially on $K_0(\scQ_{n-1})$. From Quillen's exact sequence (\ref{Q}) we see
that the base change map $\eta :K_0(\scQ_{n-1})\to K_0(\scQ)$ is an isomorphism and the result follows.
\end{proof}

We now pass to the computation of the group $K_1(\scQ)$ for a quantum torus
$\scQ$. We first remark that for an arbitrary ring $\scA$ and a unit $u\in
\scA^{\times}$  the $1 \times 1$ matrix $(u)$ is an element of $\GL_1(\scA)$.
Taking the composition of $\scA^{\times}\to \GL_1(\scA)$ with $\GL_1(\scA)\to
\GL_\infty(\scA)\to K_1(\scA)$ we obtain a canonical group homomorphism
$\lambda_\scA:\scA^{\times} \to K_1(\scA)$. In general,  $\la_\scA$ is
neither injective nor surjective, but we will show  that $\lambda_\scQ$ is
surjective when $\scQ$ is a quantum torus.

\begin{proposition}\label{K1}
Let $\scQ$ be a quantum torus. Then  $\lambda_\scQ: \scQ^{\times}\to
K_1(\scQ)$ is surjective.
\end{proposition}

\begin{proof} We argue by induction on $n\in \NN$. In case $n=0$ and $n=1$
it is well-known that $\la_\scQ$ is actually an isomorphism: $K_1(k) \simeq
k^\times$ and $K_1(k[t^{\pm 1}])\simeq k[t^{\pm 1}]^\times$ for any field $k$
by (for example) \cite[Prop.~2.2.2]{R} and \cite[Thm.~2.3.2]{R} respectively,
where both isomorphisms are induced by the determinant. Thus we can assume
$n\ge 2$ in the following.

Consider the sequence (\ref{Q}) with $\scA=\scQ_{n-1}$. We already know,
by Lemma~\ref{K0}, that $\phi^\ast$ acts trivially on $K_0(\scQ_{n-1})$ so that
we have a commutative diagram with an exact horizontal row at the bottom:
$$
\vcenter{
\xymatrix{
  \scQ_{n-1}^{\times}
  \ar[r]^{1 -\phi }\ar[d]^{\lambda_{\scQ_{n-1}}} & \scQ_{n-1}^{\times}  \ar[r]^{\eta}\ar[d]^{\lambda_{\scQ_{n-1}}} & \  \ \scQ^{\times}\ar[d]^{\lambda_\scQ}\\
     K_1(\scQ_{n-1}) \ar[r]^{1 -\phi^*} & K_1(\scQ_{n-1})  \ar[r]^{\eta^*} & K_1(\scQ)
         \ar[r]^{\partial}   & K_0(\scQ_{n-1})   \ar[r]^{} & 0
         }
         }
$$
By induction, $\lambda_{\scQ_{n-1}}$ is surjective. By Lemma~\ref{K0},
$K_0(\scQ_{n-1})\simeq\mathbb{Z}$. Clearly $\scQ^{\times}$ is generated by
$k^{\times}$ and $x_1,\ldots,x_n$,  cf.\ref{qua-pro}\eqref{qua-pro-unit}. It is
shown in the proof  of Lemma 5.16 in \cite{Q} that
$\partial(\lambda_\scQ(x_i))$ is a generator of $K_0(\scQ)\simeq \mathbb{Z}$.
To prove surjectivity of $\la_\scQ$, let $a\in K_1(\scQ)$. Then
$\partial(a)=m\in \ZZ$ and either the element $a-\lambda_Q(x_i^m)$ or  $a + \lambda_Q(x_i^m)$ lies in the kernel
of $\partial$. The claim now follows by a standard diagram chase.
\end{proof}

\begin{remark}A further diagram chase yields more than surjectivity. In fact $K_1(\scQ) = \scQ^\times/[\scQ^\times, \scQ^\times]. $ We do not need this more detailed result for our purposes.
\end{remark}

Interpreted in terms of matrices, Proposition~\ref{K1} yields the following corollary.

\begin{corollary} \label{k1-cor}
Let $\scQ$ be a quantum torus. Let $h\in\GL_\ell(\scQ)$. Then there exists a
nonnegative integer $m$  and a unit $u\in \scQ^{\times}$ such that
the matrix \begin{equation} \label{k1-cor1}
 \begin{pmatrix}
 h & 0\\
 0 &{\rm E}_m
 \end{pmatrix}  \begin{pmatrix}
 u & 0\\
 0 & \rmE_{\ell + m-1}
 \end{pmatrix}
\end{equation}
is contained in $\EL_{\ell+m}(\scQ)$.
 \end{corollary}

\subsection{Proof of the Main Theorem.} \label{pmt}
To prove the Main Theorem as stated in the introduction, we can assume that the
Cartan subalgebra $H$ of $E$ is such that $H_{cc} = \{ \sum_{i=1}^\ell s_i
E_{ii} : s_i \in k, \, \sum_i s_i = 0\} =: \frh_{\rm st}$ in the notation of
\cite{CNP}. Let $(E, H')$ be a second EALA structure,  and set $H'_{cc} =
\frh$. We then know by the main theorem of \cite{CNP} that there exists $h\in
\GL_\ell(\scQ)$ such that $\Int(h)$ maps $\frh_{\rm st}$  to $\frh$. We now
apply Corollary~\ref{k1-cor} and get $u\in \scQ^\times$ such that the matrix of
\eqref{k1-cor1} is elementary. Put \[
 g =  h \,
       \begin{pmatrix} u & 0\\ 0 & \rmE_{\ell -1 } \end{pmatrix} \in
       \GL_\ell(\scQ).
\]
Then also $\Int(g)$ maps $\frh_{\rm st}$ to $\frh$ (\cite[Lemma~2.10]{CNP}).
Moreover, \[
\begin{pmatrix}  g & 0 \\ 0 & \rmE_m \end{pmatrix} = g'
\]
is elementary. Because of \eqref{qua-pro33} we can now apply Theorem~\ref{etap}
and obtain that $\Int(g)$ lifts to an automorphism of $E$. This finishes the
proof. \hfill$\square$

\end{document}